\documentclass[12pt]{article}
    \usepackage{setspace,multirow,lscape,hhline,longtable,array} 
	\usepackage{graphicx,xspace,colortbl}
	\usepackage{amsmath,amsthm,amsfonts}
	\usepackage{color}
	\usepackage{fancybox}
	\usepackage{subfig}
	\usepackage{pdfsync}
	\usepackage[toc,page]{appendix}
	\usepackage{cite}
	\usepackage{calc}
	\usepackage[usenames,dvipsnames]{xcolor}%
	\usepackage{enumerate}
	\usepackage{changepage}
	\usepackage{diagbox}
	\usepackage{mathtools}
    \def\beq{\begin{eqnarray}}
    \def\eeq{\end{eqnarray}}
    \def\beqq{\begin{eqnarray*}}
    \def\eeqq{\end{eqnarray*}}
    
    \def\re{\textnormal {Re}}
    \def\im{\textnormal {Im}}

    \def\p{{\mathbb P}}
    \def\e{{\mathbb E}}

    \def\crp{\bar{\mathbb{R}}^+}
    \def\orp{\mathbb{R}^+}
    \def\crn{\bar{\mathbb{R}}^-}
    \def\orn{\mathbb{R}^-}

    \def\d{{\textnormal d}}

    \def\ind{{\mathbb I}}

    \def\eq{\textbf{e}(q)}
    
    \def\er{\textbf{e}(r)}

    	\def\edist{\,{\buildrel d \over =}\ }
    	
	\newtheorem{theorem}{Theorem}
	\newtheorem{lemma}{Lemma}
	\newtheorem{proposition}{Proposition}
	\newtheorem{corollary}{Corollary}
	\theoremstyle{definition}
	\newtheorem{definition}{Definition}
	\newtheorem{example}{Example}

	\newtheorem{remark}{Remark}
	
	\newcommand\xqed[1]{%
  	\leavevmode\unskip\penalty9999 \hbox{}\nobreak\hfill
  	\quad\hbox{#1}}
	\newcommand\demo{\xqed{$\dagger$}}
	\newcommand\rqed[1]{%
  	\leavevmode\unskip\penalty9999 \hbox{}\nobreak\hfill
  	\quad\hbox{#1}}
	\newcommand\rdemo{\rqed{$\ddagger$}}
\begin{document}	

\title{Wiener-Hopf Factorization for the Normal Inverse Gaussian Process}
%A recursive identity for the exponential functional
\author{Daniel Hackmann
{
%\footnote{Dept. of Mathematics and Statistics,  York University,
%4700 Keele Street, Toronto, ON, M3J 1P3, Canada.  E-mail: dan@danhackmann.com}
}}\;
 
 \date{\today}
 \maketitle

\begin{abstract}
\noindent We derive the L\'{e}vy-Khintchine representation of the Wiener-Hopf factors for the Normal Inverse Gaussian (NIG) process as well as a representation which is similar to the moment generating function (MGF) of a generalized gamma convolution (GGC). We show, via this representation, that for some parameters the Wiener-Hopf factors are, in fact, the MGFs of GGCs. Further, we develop two seperate methods of approximating the Wiener-Hopf factors, both based on Pad\'{e} approximations of their Taylor series expansions; we show how the latter may be calculated exactly to any order. The first approximation yields the MGF of a finite gamma convolution, the second that of a finite mixture of exponentials. Both provide excellent approximations as we demonstrate with numerical experiments and by considering applications to the ultimate ruin problem and to the pricing of perpetual options.
\end{abstract}
%
%{\vskip 0.15cm}
% \noindent {\it Keywords}: L\'evy processes, exponential functional
%{\vskip 0.25cm}
% \noindent {\it 2010 Mathematics Subject Classification }: Primary 60G51, Secondary 26C15 

%\newpage

%****************************************************************************************************************
%****************************************************************************************************************
%****************************************************************************************************************
\section{Introduction}
In order to determine the Wiener-Hopf factorization for the Normal Inverse Gaussian (NIG) process we are required to solve the following problem:\\
\begin{adjustwidth}{0.5cm}{0.5cm}
\emph{For } $q > 0$, $\kappa > 0$, $\sigma > 0$, $\mu \in \mathbb{R}$, $\theta \in \mathbb{R}$ \emph{, factor the expression}
\begin{align}\label{eq:initprob}
\frac{q}{q - \psi(z)},\;\;z\in\iota\mathbb{R}\quad \text{\emph{where}} \quad \psi(z) := \frac{1}{\kappa} + \frac{1}{\kappa}\sqrt{1 - 2\kappa\theta z - \kappa\sigma^2 z^2} - \mu z,
\end{align}
\emph{into the product of two functions, } $\varphi_q^+(z)$ \emph{ and } $\varphi_q^-(z)$\emph{, such that: a) } $\varphi_q^+(z)$ \emph{ is the moment generating function (MGF) of an infinitely divisible (ID) probability distribution on } $[0,\infty]$ \emph{ without drift or Gaussian component; and b) } $\varphi_q^-(z)$ \emph{, is the MGF of an ID probability distribution on $[-\infty,0]$ also without drift or Gaussian component.}\\
\end{adjustwidth}
That such a factorization exists is a well established fact, see e.g. \cite{Bertoin}, Chapter VI.2. That is, we may replace $\psi(z)$ in \eqref{eq:initprob} by the Laplace exponent of any L\'{e}vy process $X$ and be certain not only that the factorization exists, but also that 
\begin{align*}
\varphi_q^{+}(z) = \e\left[e^{zS_{\eq}}\right],\,\re(z) \leq 0, \quad \text{ and } \quad \varphi_q^-(z)=\e\left[e^{zI_{\eq}}\right],\,\re(z) \geq 0,
\end{align*}
where $S$ and $I$ are the running supremum and infimum process repectively, i.e.
\begin{align*}
S_t := \sup_{0 \leq s \leq t}X_s  \quad\text{ and }\quad I_t := \inf_{0 \leq s \leq t}X_s,
\end{align*}
and $\eq$ is an exponentially distributed random variable with mean $q^{-1}$, which is independent of $X$. Consequently, the Wiener-Hopf factorization is arguably one of the most remarkable and well-known results in the theory of fluctuations of L\'{e}vy processes.\\ \\
\noindent It is easy to see that the \emph{Wiener-Hopf factors}, i.e. $\varphi_q^{-}(z)$ and $\varphi_q^{+}(z)$, if known explicitly, are in some sense the ``next best thing" to knowing the distributions of $S_t$ and $I_t$. In fact, many practical problems involving the exit of a L\'{e}vy process (or some function thereof) from a region in the state space -- examples include the calculation of ruin probabilities (see \cite{ruin} and Section \ref{sect:ruin}), whose study originates from the insurance industry, and the pricing of financial products such as barrier options (see, e.g. \cite{barbench}) -- can be solved via the Wiener-Hopf factors. The distributions of $I_{\eq}$ and $S_{\eq}$ also appear in the pricing of perpetual options (see \cite{mordorII} and Section \ref{sect:ami}) and more generally in optimal stopping problems, see e.g. \cite{Kyprianou}, Chapter 11.\\ \\
Unfortunately, explicit, tractable expressions for $ \varphi_q^{\pm}(z)$ are not known for many processes (see Chapter 6.5 in \cite{Kyprianou} and the introduction of \cite{KuPe2011_prep} for a good overview of known Wiener-Hopf factorizations). In particular, among those classes of processes with infinite activity jumps and infinite variation paths for which there is no restriction on either positive or negative jumps, only two have known, explicit factorizations. These are: a) the stable class of processes (see  \cite{Stab1}); and b) the meromorphic class of processes (see \cite{KuzKyPa2011}).\\ \\
The methods of finding tractable expressions for the Wiener-Hopf factors and determining the distributions of $I_{\eq}$ and $S_{\eq}$ for processes with two-sided jumps fall into roughly three categories: a) by inspection; b) by solving the equivalent problem of factorizing \eqref{eq:initprob} into the product of two functions analytic and zero-free on the left and right half-planes respectively (plus a growth condition) (see \cite{Kuz2010c}, Theorem 1. (f)); c) by evaluating a general integral representation (see \cite{Kuz2010c}, Theorem 1. (b)). Method a) is only applicable in the simplest cases, e.g. when $X$ is a Brownian Motion, and the integral representation of method c) is not generally tractable, although, with some rather inspired methods, it has been used in the stable case \cite{Stab1}. Method b) is primarily useful when $\psi(z)$ is a meromorphic or rational function, as, in this case, it is possible to group poles and zeros to determine the Wiener-Hopf factors (see for example \cite{mordor} and \cite{KuzKyPa2011}). This approach is not applicable when $\psi(z)$ has branching singularities, as is the case for the NIG process as well as many other processes popular in applications (e.g. CGMY or KoBoL processes and the Variance Gamma process).\\ \\
\noindent Before describing the approach taken in this article, which differs from the above mentioned three, we take a moment to consider a special case of \eqref{eq:initprob}. Let $q = 1/\kappa$ and $\mu = 0$. In this case it is easy to derive the factorization
\begin{align*}
\frac{q}{q - \psi(z)} = \frac{1}{\sqrt{1 - 2\kappa\theta z - \kappa\sigma^2 z^2}} =  \left(1 - \frac{z}{\hat{\rho}}\right)^{-\frac{1}{2}}\left(1 - \frac{z}{\rho}\right)^{-\frac{1}{2}}
\end{align*}
where $\rho$ and $\hat{\rho}$ are just the positive and negative zeros of $p(z) := 1 - 2\kappa\theta z - \kappa\sigma^2 z^2$ respectively. We conclude that $S_{\eq}$ and $-I_{\eq}$ are gamma distributed random variables. Writing, for example,
\begin{align*}
\left(1 - \frac{z}{\rho}\right)^{-\frac{1}{2}} = \exp\left(\int_{\mathbb{R}^{+}}\log\left(\frac{u}{u-z}\right)\frac{\delta_{\rho}(\text{d}u)}{2}\right),
\end{align*}
we might conjecture that, in general, $\varphi_q^{\pm}(\pm z)$ has the form
\begin{align}\label{eq:thebasis}
\exp\left(\int_{\mathbb{R}^{+}}\log\left(\frac{u}{u-z}\right)\tau_q(\text{d}u)\right),\quad \re(z) \leq 0,
\end{align}
which, given the right restrictions on the measure $\tau_q$, would imply that the distributions of $S_{\eq}$ and $-I_{\eq}$ belong to the class of generalized gamma convolutions (GGCs) (see \cite{Bondesson}).\\ \\
\noindent It turns out that this conjecture is nearly correct. In particular, we show in Corollary \ref{cor:gegc} that $\varphi_q^{\pm}(\pm z)$ has the form \eqref{eq:thebasis}, but that $\tau_q$ is not necessarily a positive measure. Our approach in deriving this representation differs from the approaches discussed thus far. It is based on the idea that we can derive the L\'{e}vy measures of $S_{\eq}$ and $I_{\eq}$ by considering the inverse Laplace transform of the function $\Phi_q(z) := \frac{\text{d}}{\text{d}z}\log\left(\e[e^{zX_{\eq}}]\right)$. This allows us to derive the L\'{e}vy-Khintchine representation of $\varphi_q^{\pm}(z)$ in Theorem \ref{theo:main}, which leads almost directly to the representation \eqref{eq:thebasis} and an explicit formula for the measure $\tau_q$. This representation of the Wiener-Hopf factors is tractable in the sense that we are able to generate a full Taylor series expansion of $\varphi_q^{\pm}(\pm z)$ by calculating the negative of moments of $\tau_q$. This we are able to do exactly, i.e. without numerical integration, for all orders (see Section \ref{sect:techdets}). An important consequence of this fact, is that we are able to calculate Pad\'{e} Approximants (rational approximations) based on the Taylor series expansion, which yield, by virtue of an interesting connection to the theory of Stieltjes functions and an important theorem due to Rogers \cite{Rogers}, convergent approximations of $\varphi_q^{\pm}(\pm z)$ that are either: a) the MGFs of finite gamma convolutions (GCs) (when $\tau_q$ is positive); or b) the MGFs of finite exponential mixtures (MEs) (irrespective of whether $\tau_q$ is positive or not). The convergence of the approximations to $\varphi_q^{\pm}(\pm z)$ is exponential in the degree $n$ of the rational approximation, and the approximating distributions match the first $2n - 1$ moments of the distributions of $S_{\eq}$ and $-I_{\eq}$. \\ \\
\noindent This article is organized as follows: In Section \ref{sec:obs} we present some basic facts about the NIG process and state a technical lemma about solutions of the equation $q = \psi(z)$, which will be important for the remainder of the paper. It is perhaps important to note here, although details and references will be given in Section \ref{sec:obs}, that the NIG process is, in fact, a process with two-sided jumps, infinite jump activity and infinite variation paths, which is widely used for modeling both physical processes as well as economic ones. In Section \ref{sec:ggcem} we review some basic facts about GGCs, MEs and the connection between MEs and the class of L\'{e}vy processes whose L\'{e}vy measures have completely monotone densities. Section \ref{sect:padstiel} reviews the connection between GGCs, MEs and Pad\'{e} Approximants of Stieltjes functions. The main theoretical results are given in Section \ref{sect:main} in which we derive the L\'{e}vy-Khintchine representation of $\varphi_q^{\pm}(z)$ as well as the GGC-like representation \eqref{eq:thebasis}. We show that the distributions of $S_{\eq}$ and $-I_{\eq}$ belong to the class of GGCs when $\tau_q$ is a positive measure and that they do not belong to this class when $\tau_q$ is not positive. We also show that the representation \eqref{eq:thebasis} holds also for $q=0$ in the cases where this makes sense. In Section \ref{sect:techdets} we present an easy method for computing the negative moments of $\tau_q$, which are the basis for the above mentioned Taylor series expansions and Pad\'e Approximants. Finally, in Section \ref{sect:exandapp} we conduct some numerical experiments with our theoretical results and demonstrate convenient applications to the ultimate ruin problem and the pricing of perpetual stock options.\\ \\
\noindent Throughout the paper we will write $\orp$, $\crp$, $\orn$, and $\crn$, where
\begin{align*}
\orp := (0,\infty)\quad\text{ and }\quad \crp := [0,\infty),
\end{align*}
with analogous definitions for $\orn$ and $\crn$. When working with complex or imaginary numbers we will always write $\iota = \sqrt{-1}$ for the imaginary unit. As well as working with the Wiener-Hopf factors directly we will also work with the cumulant generating functions (CGFs) or Laplace exponents 
\begin{align*}
\psi_q^{\pm}(z) := \log\left(\varphi^{\pm}_q(z)\right).
\end{align*}
\section{The NIG process}\label{sec:obs}
The NIG probability distribution was first introduced by Barndorff-Nielsen in \cite{oebn}. NIG distributions form an subclass of the set of normal variance-mean mixtures, the set of generalized hyperbolic (GH) distributions \cite{oebnIII} and the set of ID distributions \cite{oebnII}. Within the class of GH distributions, the NIG distribution is the only distribution that is closed under convolutions; in general, it is a mathematically tractable version of a GH distribution that can be used to approximate the majority of GHs quite well \cite{oebnIII}. In this context it has been used to model turbulence as well as financial data. When the NIG distribution is taken as the basis for a L\'{e}vy process, it has the advantage of an explicitly defined transition density (see e.g. Table 4.5 in \cite{contan}). Additionally, its statistical properties (e.g. semiheavy tails) and the fact that NIG processes have infinite jump activity are a desirable feature when modeling stock market returns \cite{oebnII,active}. NIG processes belong to the popular class of $\mathcal{CM}$ processes, see Section \ref{sec:ggcem}, as well as to the class of regular L\'{e}vy processes of exponential type \cite{levensvet}. \\ \\
Like all L\'evy processes, we can define a NIG process $X$ via its Laplace exponent $\psi(z) := \psi_X(z) := t^{-1}\log(\exp[e^{zX_t}]$, which we will do using the parameterization found in \cite{Cont}, pg. 128\footnote{In many sources, including \cite{oebn, oebnII}, the NIG process is defined via Laplace exponent $\psi(z) = \delta\left(\sqrt{\alpha^2 - \beta^2} - \sqrt{\alpha^2 - (\beta + z^2)}\right) + \mu z$, where $0 \leq \vert \beta \vert < \alpha$, $\delta > 0$, and $\mu \in \mathbb{R}$. In this case, the process can also be defined via subordination except that the subordinator $S$ is an inverse Gaussian process with parameters $\delta$ and $\sqrt{\alpha^2 - \beta^2}$ and the Brownian motion $B$ has drift $\sqrt{\alpha^2 - \beta^2}$ and diffusion coefficient equal to one. It is easy to find a bijection between the sets of parameters $(\theta,\,\sigma,\,\kappa,\,\mu)$ and $(\alpha,\,\beta,\,\delta,\,\mu)$ and so the approaches are equivalent. The one caveat is that in the above mentioned sources the case $\alpha^2 = \beta^2$ is allowed, which would imply $\kappa = \infty$ for the parameter set used in this article. This extreme case is not included here, as it does not fit into our approach. Other authors also exclude the case $\alpha^2 = \beta^2$ in their definifions of the NIG process, see in particular \cite{contan,levensvet}. Note that in this extreme case we leave the class of processes defined by subordinating Brownian motion with a \emph{tempered} stable subordinator; in the extreme case the subordinator becomes a stable process.}, via the subordination of a Brownian motion with drift by an inverse Gaussian subordinator. Consider the subordinator $U$ with L\'{e}vy measure
\begin{align*}
\nu(\text{d}x) = \frac{1}{\sqrt{2\pi\kappa}}\frac{e^{-\frac{x}{2\kappa}}}{x^{^3/2}}\text{d}x,\quad \kappa \in \mathbb{R}^+,
\end{align*}
and note that with this parameterization $\kappa$ is in fact the variance of $U$. The Laplace exponent of the process $U$ is then
\begin{align}\label{eq:sub}
\psi_U(z) := \frac{1}{t}\log\left(\e[e^{zU_t}]\right) = \frac{1}{\kappa} - \frac{1}{\kappa}\sqrt{1 - 2\kappa z},\quad \re(z) < \frac{1}{2\kappa} .
\end{align}
Subordinating the Brownian motion with drift, $B$, with Laplace exponent
\begin{align*}
\psi_B(z) := \frac{1}{t}\log\left(\e[e^{zB_t}]\right) = \theta z + \frac{1}{2}\sigma^2 z^2,\quad z \in \mathbb{C},\,\theta \in \mathbb{R},\, \sigma \in \mathbb{R}^+,
\end{align*}
by the process $U$ gives us the Laplace exponent of a NIG process $X^{(0)}$ without drift, i .e.
\begin{align*}
\psi_{X^{(0)}}(z) := \psi_U(\psi_B(z)) = \frac{1}{\kappa} - \frac{1}{\kappa}\sqrt{1 - 2\kappa\theta z - \kappa\sigma^2 z^2}, \quad \hat{\rho} < \re(z) < \rho,
\end{align*}
where 
\begin{align}\label{eq:rhos}
\rho := \frac{-\theta + \sqrt{\theta^2 + \frac{\sigma^2}{\kappa}}}{\sigma^2},\qquad \hat\rho := \frac{-\theta - \sqrt{\theta^2 + \frac{\sigma^2}{\kappa}}}{\sigma^2},
\end{align}
such that $p(z) = (1 - z/\rho)(1 - z/\hat{\rho}) = 1 - 2\kappa\theta z - \kappa\sigma^2z^2$.
Adding a drift $\mu \in \mathbb{R}$ to $X^{(0)}$ gives a general NIG process $X$ with parameters $(\theta,\,\sigma,\,\kappa,\,\mu)$, which has Laplace exponent
\begin{align}\label{eq:psiX}
\psi_X(z) := \psi_{X^{(0)}}(z) + \mu z =  \frac{1}{\kappa} - \frac{1}{\kappa}\sqrt{1 - 2\kappa\theta z - \kappa\sigma^2 z^2} + \mu z, \quad \hat{\rho} < \re(z) < \rho,
\end{align}
where the funtion on the righthand side of \eqref{eq:psiX} can be extended to an analytic function in the cut complex plane $\mathbb{C} \backslash (-\infty,\hat{\rho}] \cup [\rho,\infty)$. If $\eq$ is an exponential random variable with mean $q^{-1}$, which is independent of $X$, then we have
\begin{align}\label{eq:tofactor}
\e[e^{zX_{\eq}}] = \frac{q}{q-\psi_X(z)} = \frac{q}{q - \frac{1}{\kappa} + \frac{1}{\kappa}\sqrt{1 - 2\kappa\theta z - \kappa\sigma^2 z^2} - \mu z}, 
\end{align}
where the equalities hold on some non-empty, vertical strip in the complex plane containing zero. The righthand side of \eqref{eq:tofactor} is again a well-defined and analytic function of $z$ on $\mathbb{C} \backslash (-\infty,\hat{\rho}] \cup [\rho,\infty)$ except at those points where
\begin{align}\label{eq:main}
\psi_X(z) = q.
\end{align}
It is easy to see that if solutions of \eqref{eq:main} exist, they will have the form
\begin{gather*}
\zeta := \zeta(q) = \frac{-\theta -\mu +\kappa  \mu  q +\sqrt{d}}{\kappa  \mu ^2+\sigma ^2}, \quad\text{ and }\quad \hat\zeta := \hat{\zeta}(q) = \frac{-\theta -\mu +\kappa  \mu  q -\sqrt{d}}{\kappa  \mu ^2+\sigma ^2}, \quad\text{where}\\
d := \theta ^2+\mu ^2-2 \theta  \mu  (q\kappa-1)+q \sigma ^2 (2-q\kappa).
\end{gather*}
In fact, $\zeta$ and $\hat{\zeta}$ are just the solutions of the associated quadratic equation 
\begin{align}\label{eq:assocquad}
p(z) = [r(z)]^2,
\end{align}
where $r(z) := 1 -q\kappa  + \kappa \mu z$. The following technical proposition is important in helping us determine the number, mutiplicity, and location of solutions of \eqref{eq:main} and will be referenced throughout the article. Its proof is straightforward and a little tedious, so we relegate it to Appendix \ref{sec:appendix}.
\begin{proposition}\label{prop:themostannoying}\
\begin{enumerate}[(i)]
\item Equation \eqref{eq:main} has either no solutions, one solution, or two solutions. 
\item $z_0 \in \mathbb{C}$ is a solution of \eqref{eq:main} iff $z_0 \in [\hat{\rho},0)\cup(0,\rho]$,  $z_0 = \zeta$ or $z_0 = \hat{\zeta}$, $d > 0$, and $q - 1/\kappa \leq \mu z_0$. It follows that whenever one of $\zeta$ or $\hat{\zeta}$ is a solution of \eqref{eq:main} $\zeta \neq \hat{\zeta}$.
\item If $\zeta$ (resp. $\hat{\zeta}$) satisfies \eqref{eq:main}, then  $0 < \zeta$ (resp. $\hat{\zeta}< 0$). It follows from (ii) that if $\zeta$ and $\hat{\zeta}$ satisfy \eqref{eq:main}, then $\hat{\rho} \leq \hat{\zeta} < 0 < \zeta \leq \rho$.
\item If $z_0$ satisfies \eqref{eq:main} and $\hat{\rho} < z_0 < \rho$, then $z_0$ is a simple zero of $q - \psi_X(z)$.
\item If $\zeta \in \mathbb{R}$ (resp. $\hat{\zeta} \in \mathbb{R}$) then $\hat{\rho} \leq \zeta \leq \rho$ (resp. $\hat{\rho} \leq \hat{\zeta} \leq \rho$).
\item Neither $\rho = \hat{\zeta}$ nor $\hat{\rho} = \zeta$ is possible.
\item Both $\rho = \zeta$ and $\hat{\rho} = \hat{\zeta}$ iff $\mu = 0$ and $q = 1/\kappa$.
\end{enumerate}
\end{proposition}
\section{Generalized gamma convolutions, exponential mixtures, and the $\mathcal{CM}$ class of processes}\label{sec:ggcem}
In this brief section we review some facts about generalized gamma convolutions, exponential mixtures, and L\'{e}vy processes whose jumps are determined by L\'{e}vy measures with completely monotone densities. We denote this latter group of processes by $\mathcal{CM}$. The content in this section is taken primarily from Bondesson \cite{Bondesson} and Rogers \cite{Rogers}. Going forward we will write $\Gamma(\alpha,\beta)$ for the Gamma distribution with density
\begin{align*}
f(x) = \frac{\beta^{\alpha}}{\Gamma(\alpha)}x^{\alpha - 1}e^{-\beta x},\quad x > 0.
\end{align*}
\begin{definition}\label{def:ggc}
A \emph{generalized gamma convolution} is a probability distribution on $\bar{\mathbb{R}}^+$ with MGF
\begin{align*}
\varphi(z) = \exp\left(az + \int_{\orp}\log\left(\frac{u}{u-z}\right)\tau(\text{d}u)\right), \quad \re(z)\leq 0,
\end{align*}
where $a \geq 0$ and $\tau$ is a radon measure on $\mathbb{R}^+$ satisfying
\begin{align*}
\int_{(0,1]}\vert\log(u)\vert \tau(\text{d}u) < \infty,\quad\text{ and }\quad \int_{(1,\infty)}\frac{\tau(\text{d}u)}{u} < \infty.
\end{align*}
\end{definition}
\noindent The measure $\tau$ is referred to as the \emph{Thorin measure} and the name generalized gamma convolution is easy to justify given that a convolution of a finite number of independent gamma distributions is a special case of a GGC with a Thorin measure that has finite support. The inclusion of the constant $a$ in the definition owes to the fact that the distribution $\Gamma(\beta,\beta/a)$ converges weakly to the degenerate distribution at the point $a$ as $\beta \rightarrow \infty$. For our purposes it is important to note that: a) an arbitrary GGC is the weak limit of a sequence of convolutions of finite numbers of independent gamma distributions (see Theorem 3.1.5 in \cite{Bondesson}); and b) that GGCs are ID distributions such that the following relationship holds.
\begin{theorem}[Theorem 3.1.1 in \cite{Bondesson}]\label{theo:bondy}
A probability distribution on $\crp$ is a GGC iff it is an ID distribution whose L\'{e}vy measure has a density $\pi(x),\, x > 0$, such that $x\pi(x)$ is a completely monotone function. In this case, we have the following relationship between the L\'{e}vy density and the Thorin measure
\begin{align*}
\pi(x) = \frac{1}{x}\int_{\orp}e^{-xu}\tau(\text{d}u).
\end{align*}
\end{theorem}
\noindent We recall that a \emph{completely monotone function} $f(x)$, defined for $x > 0$, is a smooth function that satisfies
\begin{align*}
(-1)^nf^{(n)}(x) \geq 0,\quad n \in \mathbb{N} \cup \{0\},
\end{align*}
and that by Bernstein's theorem, every completely monotone function has a representation of the form $f(x) = \int_{\crp}e^{-xu}\mu(\text{d}u)$ for a measure $\mu$ on $\crp$. 
\begin{definition}\label{def:compmon}
A process L\'e{v}y process $X \in \mathcal{CM}$ if its L\'{e}vy measure $\Pi$ has the form
\begin{align*}
\Pi(\text{d}x) = \left(\ind(x > 0)\int_{\mathbb{R}^+}e^{-xu}\mu^+(\text{d}u) + \ind(x < 0)\int_{\mathbb{R}^+}e^{xu}\mu^-(\text{d}u)\right)\text{d}x,
\end{align*}
for measures $\mu^+$ and $\mu^-$, which satisfy
\begin{align*}
\int_{\mathbb{R}^+}\frac{1}{u(1 + u)^2}(\mu^+ + \mu^-)(\text{d}u) < \infty.
\end{align*}
\end{definition}
\noindent That every NIG process belongs to $\mathcal{CM}$ is perhaps not obvious from the discussion thus far, however, it is clear from the L\'{e}vy-Khintchine representation of the Laplace exponent, which was first derived by Barndorff-Nielsen in \cite{oebn} and then again more directly in \cite{oebnII}.
\begin{definition}\label{def:mixture}
A probability distribution on $\crp$ is a \emph{mixture of exponentials} if its MGF has the form
\begin{align}\label{eq:stiel}
\varphi(z) = \int_{(0,\infty]}\frac{u}{u-z}\mu(\text{d}u), \quad \re(z) \leq 0,
\end{align}
where $\mu$ is a probabilty distribution on $(0,\infty]$. A \emph{finite mixture of exponentials} results when $\mu$ has finite support.
\end{definition}
\noindent Note that while all non-degenerate GGCs are absolutely continuous with respect to the Lebesgue measure, this is not the case for MEs, since the measures $\mu$ can have an atom at $\infty$. In this case an ME will have an atom at zero and a density with respect to the Lebesgue measure on $(0,\infty)$ (see also discussion on pg. 25 and 30 in \cite{Bondesson}). \\ \\
\noindent What ties MEs and the $\mathcal{CM}$ class of processes together is the following important theorem due to Rogers \cite{Rogers}.
\begin{theorem}[Theorem 2 in \cite{Rogers}]\label{theo:rog}\ \\
\begin{enumerate}[(i)]
\item If $X \in \mathcal{CM}$, then the distributions of $S_{\eq}$ and $-I_{\eq}$ are MEs for each $q > 0$.
\item  If the distributions of $S_{\eq}$ and $-I_{\eq}$ are MEs for some $q > 0$, then $X \in \mathcal{CM}$.
\end{enumerate}
\end{theorem}
\section{Pad\'{e} approximants of Stieltjes functions and the connection to GCCs and MEs}\label{sect:padstiel}
This section is the companion to Section \ref{sec:ggcem} in the sense that we present a very natural and elegant way to approximate general GGCs and MEs by gamma convolutions and finite MEs respectively. The technique in both cases relies on Stieltjes functions and their Pad\'{e} approximants. 
\begin{definition}
A \emph{Stieltes function} $f(z)$ is defined by the Stieltjes-integral representation,
\begin{align*}
f(z) := \int_{\crp} \frac{\mu(\text{d}u)}{1 + zu},\quad z \in \mathbb{C}\backslash(-\infty,0],
\end{align*}
where $\mu$ is a positive measure on $\crp$ with infinite support and finite moments
\begin{align*}
m_k := \int_{\crp}u^k \mu(\text{d}u).
\end{align*}
\end{definition}
\noindent Formally, we may also express $f(z)$ as a \emph{Stieltjes series}, which may converge only at $0$ and has the following form:
\beq\label{eq:stielser}
f(z) = \sum_{k=0}^{\infty} m_k(-z)^k.
\eeq
It is easy to see that the above series converges for $|z|<R$ if and only if the support of $\mu$ lies in $[0,R^{-1}]$. In this case we will call $f(z)$ a \emph{Stieltjes function (or a Stieltjes series) with the radius of convergence} $R$. For such functions, the domain of definition extends to all $z \in \mathbb{C}\backslash(-\infty,-R]$.\\ \\
\noindent For the following, we assume $f(z)$ is a function (not necessarily a Stieltjes function) with a power series representation $f(z) = \sum_{k=0}^{\infty}c_kz^k$ at zero. 
\begin{definition}
If there exist polynomials $P_m(z)$ and $Q_n(z)$ satisfying $\deg(P_m)\le m$, $\deg(Q_n)\le n$, $Q_n(0) = 1$ and
\begin{align*}
\frac{P_m(z)}{Q_n(z)} = f(z) + O(z^{m+n+1}), \;\;\; z\to 0,
\end{align*}
then we say that $f^{[m/n]}(z) := P_m(z)/Q_n(z)$ is the \emph{$[m/n]$ Pad\'{e} approximant} of the function $f(z)$ (at zero).
\end{definition}
\noindent The connection between Stieltjes functions and Pad\'{e} approximants is nicely summarized in the following theorem due to Baker \cite{Baker}.
\begin{theorem}[Corollary 5.1.1, and Theorems 5.2.1, 5.4.4 in \cite{Baker}]\label{theo:exist} If $f(z)$ is a Stieltjes function with radius of convergence $R > 0$, then $f^{[n + k/n]}(z)$ exists provided $k \geq -1$ and $n \geq 2$. The approximant $f^{[n+k/n]}(z)$ has simple poles in $(-\infty,-R]$, which have positive residues. Further, on any compact subset $S \subset z \in \mathbb{C}\backslash(-\infty,-R]$ 
\begin{align*}
\vert f(z) - f^{[n+k/n]}(z)\vert < c_1e^{-c_2 n},
\end{align*}
where $c_1 := c_1(k,S)$ and $c_2 := c_2(S)$ are both greater than zero.
\end{theorem}
\noindent To make the connection with GGCs, let us assume $\varphi(z)$ is the MGF of a GGC with corresponding random variable $Y$ and Thorin measure $\tau$ with infinite support. We further assume that $\varphi(z)$, and therefore also $\psi(z) := \log(\varphi(z))$, is analytic at zero such that $\psi(z)$ has radius of convergence $R$. For simplicity we assume the constant $a$ in Definition \ref{def:ggc} is zero. Further, from here on, we will denote the pushforward measure of $\mu$ under the transformation $x \mapsto x^{-1}$ by $^*\mu$.
\begin{lemma}\label{lem:ggcstiel}
The function $\psi'(-z)$ is a Stieltjes function with radius of convergence $R$, in particular $\psi'(z)$ has the following analytic continuation to the cut complex plane
\begin{align*}
\psi'(z) = \int_{(0,R^{-1}]}\frac{u^*\tau(\text{d}u)}{1-zu},\quad z \in \mathbb{C}\backslash[R,\infty).
\end{align*}

\end{lemma}
\begin{proof}
The assumptions that $\psi(z)$ is analytic with radius of convergence $R$ implies $\tau$ has support in $[R,\infty)$. Therefore, for $z \notin [R,\infty)$
\begin{align*}
\psi'(z) = \frac{\text{d}}{\text{d}z}\int_{[R,\infty)}\log\left(\frac{u}{u - z}\right)\tau(\text{d}u) = \int_{[R,\infty)}\frac{\tau(\text{d}u)}{u - z} = \int_{(0,R^{-1}]}\frac{u\ \!\! ^*\tau(\text{d}u)}{1-uz},
\end{align*}
where we have made the change of variables $u \mapsto u^{-1}$ in the last step. That the measure $u\ \!\! ^*\tau(\text{d}u)$ has moments of all orders is guaranteed by the conditions imposed on $\tau$ in Definition \ref{def:ggc} and the fact that the support of $\tau$ lies in  $[R,\infty)$.
\end{proof}
\begin{proposition}\label{prop:ggc}
For all $n \geq 2$ the function
\begin{align*}
\varphi_n(z) = \exp\left(\int_0^z(\psi')^{[n-1/n]}(w)\text{d}w\right),\quad \re(z) < R,
\end{align*}
is the MGF of a convolution of $n$ independent gamma distrubtions. The corresponding random variable $Y^{(n-1)}$ has the property
\begin{align*}
\e[(Y^{(n-1)})^k] = \e[Y^k],\quad 1 \leq k \leq 2n-1. 
\end{align*}
\end{proposition}
\begin{proof}
According to Theorem \ref{theo:exist} we have
\begin{align*}
(\psi')^{[n-1/n]}(-w) = \sum_{i=1}^{n}\frac{\alpha_i}{w + \beta_i},
\end{align*}
where $\alpha_i > 0$ and $\beta_i > R$ for all $1\leq i \leq n$. As a result
\begin{align*}
\varphi_n(z) = \exp\left(\sum_{i=1}^n\int_0^{z}\frac{\alpha_i}{\beta_i - w}\text{d}w\right) = \prod_{i=1}^{n}\left(1 - \frac{z}{\beta_i}\right)^{-\alpha_i},
\end{align*}
which demonstrates the first part of the claim. For the second part, we simply need to observe that by defintions of Pad\'{e} approximations and Stieltjes functions we have
\begin{align*}
\psi'(z) = \sum_{k=0}^{\infty}m_kz^{k}\quad\text{ and }\quad \log(\varphi_n(z)) = \int_{0}^z\sum_{k=0}^{\infty}c_kw^{k}\text{d}w,\quad \vert z \vert < R,
\end{align*}
where $m_k = c_k$ for $0 \leq k \leq 2n-1$. From here, it is easy to see that the first $2n-1$ cumulants of $Y$ and $Y^{(n-1)}$, and therefore also the first $2n-1$ moments, are identical.
\end{proof}
\noindent A connection between Stieltjes functions and Pad\'e approximations and MEs is also easy to establish. In what follows suppose that $Y$ is a random variable whose distribution is a ME with MGF $\varphi(z)$ such that the measure $\mu$ in Definition \ref{def:mixture} has infinite support. Further assume that $\varphi(z)$ is analytic at zero such that its power series has radius of convergence $R$.
\begin{lemma}\label{lem:me}
The function $\varphi(-z)$ is a Stieltjes function with radius of convergence $R$.
\end{lemma}
\begin{proof}
The assumption that $\varphi(z)$ is analystic with radius of convergence $R$ implies that the measure $\mu$ in the in Definition \ref{def:mixture} has support in $[R,\infty]$. Applying the change of variables $x \mapsto x^{-1}$ in the integral \eqref{eq:stiel} gives the result.
\end{proof}
\noindent This leads us directly to the result which is the analogy for MEs to Proposition \ref{prop:ggc}.
\begin{proposition}\label{prop:me}
For $\re(z) < R$, the function $\varphi^{[n-1/n]}(z)$ (resp. $\varphi^{[n/n]}(z)$) is a MGF of a finite mixtures of exponentials. The corresponding random variable $Y^{(n-1)}$ (resp. $Y^{(n)}$)  has the property
\begin{align*}
\e[(Y^{(n-1)})^k] = \e[Y^k],\quad 1 \leq k \leq 2n-1,\quad \left(\text{resp. } \e[(Y^{(n)})^k] = \e[Y^k],\quad 1 \leq k \leq 2n\right).
\end{align*}
\end{proposition}
\begin{proof}
Follows immediately from Lemma \ref{lem:me} and Theorem \ref{theo:exist}.
\end{proof}
\begin{remark}
Recall that the MGF of an ME distributed random variable $Y$ has the form
\begin{align*}
\int_{(0,\infty]}\frac{u}{u-z}\mu(\text{d}u),
\end{align*}
where $\mu$ is a probability distribution, which may have an atom with weight $a$ at $\infty$. If this is the case, then the distribution of $Y$ will have an atom with weight $a$ at zero. In choosing an approximation by a finite mixture of exponentials via the Pad\'{e} approximation, it is clear from Proposition \ref{prop:me} that we can adjust the approximation to be either absolutely continuous (by choosing the $[n-1/n]$ approximation) or to have an atom at zero (by choosing the $[n/n]$ approximation). Further, according to Theorem \ref{theo:rog}, and the fact that NIG processes belong to $\mathcal{CM}$, the distribution of $S_{\eq}$ will be an ME, and it is easy to show that $\p[S_{\eq} = 0] = 0$, which holds essentially because the NIG process is an infinite variation process. Therefore, the distribution of $S_{\eq}$ will be absolutely continuous, and we will focus only on the $[n-1/n]$ approximation in this article. The same is true of $-I_{\eq}$, of course.\demo
\end{remark}
% By the final value theorem...
 
\noindent We end this section with a brief description of the computation of the coefficients of $[n-1/n]$ Pad\'{e} approximants. For the interested reader, the book \cite{Baker} by Baker is a good source for information on Pad\'{e} approximants in general. Consider a function $f(z)=\sum_{k = 0}^{\infty} c_k z^k$, whose $[n-1/n]$ Pad\'{e} approximant $f^{[n-1/n]}(z)=P_n-1(z)/Q_n(z)$ is known to exist. First, we solve the system of $n$ linear equations 
\beq\label{find_b}
\left( \begin{array}{ccccc}
c_{0} & c_{1} & c_{2} & \cdots & c_{n-1} \\
c_{1} & c_{2} & c_{3} &\cdots & c_{n} \\
c_{2} & c_{3} & c_{4} &\cdots & c_{n+1} \\
\vdots & \vdots & \vdots & \ddots & \vdots  \\
c_{n-1} & c_{n} & c_{n+1} &\cdots & c_{2n-1} 
\end{array} \right)
\left( \begin{array}{c}
b_n \\ b_{n-1} \\ b_{n-2} \\ \vdots \\ b_1
\end{array} \right)
=-
\left( \begin{array}{c}
c_{n} \\ c_{n+1} \\ c_{n+2} \\ \vdots \\ c_{2n}
\end{array} \right)
\eeq
whose solutions $b_k$, $1\le k \le n$, give us the coefficients of the denominator $Q_n(z):=1+b_1z\\+b_2z^2+\cdots+b_nz^n$. Then, the coefficients of the numerator
$P_{n-1}(z):=a_0+a_1z+a_2z^2+\cdots+a_{n-1}z^{n-1}$ can be calculated as follows:
\beq\label{compute_coefficients_a} \nonumber
&&a_0=c_0, \\ \nonumber
&&a_1=c_1+b_1c_0, \\
&&a_2=c_2+b_1c_2+b_2c_0,\\ \nonumber
&& \;\;\;\;\;\vdots\\ \nonumber
&&a_{n-1}=c_{n-1}+\sum\limits_{k=1}^n b_k c_{n-1-k}. 
\eeq
In practice, when $n$ is even moderately large, the system in \eqref{find_b} will have a very large condition number, and solving the system of linear equations \eqref{find_b} will likely involve a loss of accuracy. This can be avoided by using higher precision arithmetic. For the computations in this article we use Mathematica, which supports arbitrary precision arithmetic, as well as the MPFUN90 arbitrary precision package for Fortran-90 \cite{Bailey95afortran90}.
\section{The Wiener-Hopf factorization for the NIG process}\label{sect:main}
Before we state the main results, we consider a general method for determining the Wiener-Hopf factors, which will be used in the proof of Theorem \ref{theo:main}, but is theoretically valid for those L\'{e}vy processes whose Laplace exponents are analytic at zero. Let $X$ be a L\'{e}vy process and let $S$ (resp. $I$) be the running supremum (resp. infimum) process. The standard Wiener-Hopf theory for L\'{e}vy process (see for example Theorem 6.15 in \cite{Kyprianou}) then shows: a) $S_{\eq}$ and $-I_{\eq}$ are positive, ID random variables without drift or Gaussian component; and b) $X_{\eq} \edist S_{\eq} + I_{\eq}$. Therefore, we must have
\begin{align*}
\e[e^{zX_{\eq}}] = \exp\left(\int_{\mathbb{R}}\left(e^{zx} - 1\right) \Pi_q(\text{d}x )\right)
\end{align*}
for some L\'{e}vy measure $\Pi_q$ on $\mathcal{B}_{\mathbb{R}}$, which satisfies the condition
\begin{align}\label{eq:finmeas}
\int_{\mathbb{R}}\min(1,\vert x \vert)\Pi_q(\text{d}x) < \infty.
\end{align}
It follows that
\begin{align}\label{eq:lkwhf}
\varphi_q^+(z) = \exp\left(\int_{\mathbb{R}^+}\left(e^{zx} - 1\right) \Pi_q^+(\text{d}x )\right)\,\,\text{and}\,\,\varphi_q^-(z) = \exp\left(\int_{\mathbb{R}^-}\left(e^{zx} - 1\right) \Pi_q^-(\text{d}x )\right),
\end{align}
where $\Pi^+_q$ is the measure $\Pi_q$ restricted to $\mathbb{R}^+$ and  $\Pi^-_q$ is the measure $\Pi_q$ restricted to $\mathbb{R}^-$. In what follows, the central idea is to determine $\Pi_q$ by inversion of the Laplace transform, from which it is straightforward to identify $\Pi^+_q$ and $\Pi^-_q$ and therefore derive the explicit L\'{e}vy-Khintchine representation \eqref{eq:lkwhf} of the Wiener-Hopf factors.\\ \\
\noindent To this end we make some observations, which are either well-known facts, or have straightforward proofs:
\begin{enumerate}[(O1)]
\item If $\psi_X(z)$ is analytic at zero, then
\begin{align}\label{eq:impobs}
\Phi_q(z) := \frac{\text{d}}{\text{d}z}\log\left(\e[e^{zX_{\eq}}]\right) = \int_{\mathbb{R}}e^{zx}x\Pi_q(\text{d}x),
\end{align}
which is finite at least on some strip $S_{\alpha} = \{z \in \mathbb{C}\,:\, -\alpha < \re(z) < \alpha \}$, $\alpha := \alpha(q) > 0$.
\item It follows from \eqref{eq:finmeas} and (O1) that the measure $x\Pi_q(\text{d}x)$ is a finite, signed measure.
\item From (O2) it follows that the function $F_q(t) := \int_{(-\infty,t]}x\Pi_q(\text{d}x)$ is a right continuous function of bounded variation with the property $F_q(-\infty) = 0$ (see \cite{folland}, pg. 104, Theorem 3.29).
\item From (O1) it follows that $F_q(t) = o(e^{C t})$ as $t \rightarrow -\infty$ for $-\alpha < C < 0$ and $F(\infty) - F(t) = o(e^{-C t})$ as $t \rightarrow \infty$ for $0 < C < \alpha$, from which, together with (O3), it follows, via integration by parts, that
\begin{align*}
-\frac{\Phi_q(z)}{z} &= \int_{\mathbb{R}}e^{zt}F_q(t)\text{d}t,\quad -\alpha < \re(z) < 0, \quad \text{ and }\\
 \quad  \frac{\Phi_q(z)}{z} &= \int_{\mathbb{R}}e^{zt}\left(F_q(\infty) - F_q(t)\right)\text{d}t,\quad 0 < \re(z) < \alpha.
\end{align*}
\item From (O4) it follows that
\begin{align}
 \frac{F_q(t+) + F_q(t -)}{2} &= \frac{1}{2\pi \iota}\int_{C + \iota\mathbb{R}}e^{-tz}\left(-\frac{\Phi_q(z)}{z}\right)\text{d}z,\quad -\alpha < C < 0,\quad\text{ and } \label{eq:negt}\\
F_q(\infty) -  \frac{F_q(t+) + F_q(t -)}{2} &= \frac{1}{2\pi \iota}\int_{C + \iota\mathbb{R}}e^{-tz}\left(\frac{\Phi_q(z)}{z}\right)\text{d}z,\quad 0 < C < \alpha, \label{eq:post}
\end{align}
(see \cite{doetsch}, pg. 169, Satz 24.3).
\item From (O5) and (O3) it follows that if $G_q(t) := \frac{1}{2}(F_q(t+) + F_q(t -))$ is continuous at $t$, then $F_q(t)$ is also continuous at $t$ and $F_q(t) = G_q(t)$. I.e. the procedure in (O5) actually returns the original function values $F_q(t)$, or $F_q(\infty) - F_q(t)$, wherever $G_q(t)$ is continuous.
\item If, in addition, we can determine $F_q'(t)$ almost everywhere (w.r.t. the Lebesgue measure), and $F_q(t) = \int_{-\infty}^{t}F_q'(x)\text{d}x$, then $x\Pi_q(\text{d}x) = F_q'(x)\text{d}x$.
\end{enumerate}
We now use the above described approach for the NIG process. The reader should assume that the notation $X$, $\psi_X(z)$, $S_{\eq}$, and $I_{\eq}$ refers to a NIG process for the remainder of this section. We will see shortly that the form of the L\'{e}vy-Khintchine representation of the Wiener-Hopf factors of the NIG process depends on: a) whether or not $\zeta$ and $\hat{\zeta}$ are solutions of \eqref{eq:main}, i.e. of $q = \psi_X(z)$; and b) whether or not $\zeta = \rho$ or $\hat{\zeta} = \hat{\rho}$ (see Section \ref{sec:obs} and Proposition \ref{prop:themostannoying} for definitions and properties of $\zeta$, $\hat{\zeta}$, $\rho$, and $\hat{\rho}$) . Let us define the following cases for $\zeta$
\begin{enumerate}[I:]
\item $\zeta$ is not a solution of \eqref{eq:main} and $\zeta \neq \rho$
\item $\zeta$ is a solution of  \eqref{eq:main} and $\zeta \neq \rho$
\item  $\zeta = \rho$.
\end{enumerate}
Similarly, for $\hat{\zeta}$ we define
\begin{enumerate}[A:]
\item $\hat{\zeta}$ is not a solution of \eqref{eq:main} and $\hat{\zeta} \neq \hat{\rho}$
\item $\hat{\zeta}$ is a solution of \eqref{eq:main} and $\hat{\zeta} \neq \hat{\rho}$
\item $\hat{\zeta} = \hat{\rho}$.
\end{enumerate}
Next, let us define the (not necessarily positive) measures
\begin{gather}\label{eq:themus}
\mu_q^+(\text{d}u) := \ind(u > \rho)\frac{a(bu - c)}{\pi(u - \zeta)(u - \hat{\zeta})\sqrt{(u - \rho)(u - \hat{\rho})}}\text{d}u\\
\mu_q^-(\text{d}u) := \ind(u < \hat{\rho})\frac{a(bu - c)}{\pi(u - \zeta)(u - \hat{\zeta})\sqrt{(u - \rho)(u - \hat{\rho})}}\text{d}u,\nonumber
\end{gather}
where
\begin{align}\label{eq:const}
a := \frac{1}{\sigma\kappa^{3/2}(\mu^2 + \sigma^2/\kappa)},\quad\quad b:= \theta\mu\kappa + (q\kappa - 1)\sigma^2,\quad\quad c:= \mu-\theta(q\kappa-1).
\end{align}
Similarly, we define
\begin{gather}
\nu_q^+(\text{d}u) := \ind(u > \rho)\frac{ab}{\pi(u - \hat{\zeta})\sqrt{(u - \rho)(u - \hat{\rho})}}\text{d}u, \quad \nu_q^-(\text{d}u) := \ind(u < \hat{\rho})\frac{ab}{\pi(u - \hat{\zeta})\sqrt{(u - \rho)(u - \hat{\rho})}}\text{d}u,\label{eq:thenus}\\
\lambda_q^+(\text{d}u) := \ind(u > \rho)\frac{ab}{\pi(u - \zeta)\sqrt{(u - \rho)(u - \hat{\rho})}}\text{d}u,\quad \lambda_q^-(\text{d}u) := \ind(u < \hat{\rho})\frac{ab}{\pi(u - \zeta)\sqrt{(u - \rho)(u - \hat{\rho})}}\text{d}u,\label{eq:thelams}
\end{gather}
where $a$ and $b$ are as in \eqref{eq:const}. \\ \\
\noindent With these definitions we can give our first main result, the L\'{e}vy-Kintchine representation of the Wiener-Hopf factors of the NIG process.
\begin{theorem}\label{theo:main}
For the NIG process, the measures $\Pi^+_q$ and $\Pi^-_q$ are absolutely continuous with respect to the Lebesgue measure with densities
\begin{align*}
\pi^+_q(x) = \ind(x > 0)\frac{1}{x}\int_{\mathbb{R}^+}e^{-xu}\omega_q^+(\text{d}u)\quad\text{ and }\quad \pi^-_q(x) = \ind(x < 0)\frac{1}{x}\int_{\mathbb{R}^-}e^{-xu}\omega_q^-(\text{d}u),
\end{align*}
where the forms of $\omega_q^+$ and $\omega_q^-$ are case dependent and are given in Table \ref{tab:meas}.
\renewcommand{\arraystretch}{3.25}
\begin{table}
\begin{center}
\begin{tabular}{ | c || c | c | c |}
  \hline 
\diagbox{$\mathbf{Case}$}{$\mathbf{Case}$} & I & II & III \\
\hline \hline 
A & {$\!\begin{aligned} 
               \omega_q^+ &= \mu_q^+  \\ 
               \omega_q^- &= \mu_q^- \end{aligned}$} & {$\!\begin{aligned} 
               \omega_q^+ &= \mu_q^+ + \delta_{\zeta}\\ 
               \omega_q^- &= \mu_q^- \end{aligned}$} &{$\!\begin{aligned} 
               \omega_q^+ &= \nu_q^+ + \frac{1}{2}\delta_{\rho} \\ 
               \omega_q^- &= \nu_q^-\end{aligned}$} \\[2ex]
\hline
B &{$\!\begin{aligned} 
               \omega_q^+ &= \mu_q^+ \\ 
               \omega_q^- &= \mu_q^- - \delta_{\hat{\zeta}}\end{aligned}$} &{$\!\begin{aligned} 
               \omega_q^+ &= \mu_q^+ + \delta_{\zeta} \\ 
               \omega_q^- &= \mu_q^- - \delta_{\hat{\zeta}}\end{aligned}$}  & {$\!\begin{aligned} 
               \omega_q^+ &= \nu_q^+ + \frac{1}{2}\delta_{\rho} \\ 
               \omega_q^- &= \nu_q^- - \delta_{\hat{\zeta}}\end{aligned}$}\\[2ex]
\hline
C &{$\!\begin{aligned} 
               \omega_q^+ &= \lambda_q^+ \\ 
               \omega_q^- &= \lambda_q^- -\frac{1}{2}\delta_{\hat{\rho}}\end{aligned}$} & {$\!\begin{aligned} 
               \omega_q^+ &= \lambda_q^+ +\delta_{\zeta} \\ 
               \omega_q^- &= \lambda_q^- -\frac{1}{2}\delta_{\hat{\rho}}\end{aligned}$} & {$\!\begin{aligned} 
               \omega_q^+ &= \frac{1}{2}\delta_{\rho} \\ 
               \omega_q^- &= -\frac{1}{2}\delta_{\hat{\rho}} \end{aligned}$} \\[3ex]
\hline
\end{tabular}
\caption{Form of the measures $\omega_q^+$ and $\omega_q^-$.}\label{tab:meas}
\end{center}
\end{table}
\end{theorem}

\begin{proof}
Our goal will be to determine the function $F_q(t)$ from the preceding discussion and its derivative. To do this, we will derive an expression for the function $G_q(t)$ via the Formulas \ref{eq:negt} and \ref{eq:post}. Specifically, we will derive an expression for $G_q^{-}(t) := \ind(t < 0)G_q(t)$ using Formula \ref{eq:negt} and an expression for $G_q^{+}(t) := \ind(t>0)G_q(t)$ using Formula \ref{eq:post} for the cases I-A, II-A, and III-A. The other cases can be treated in an analogous manner. \\ \\
\noindent To begin, note that
\begin{align}\label{eq:phiq}
\Phi_q(z) = \frac{\text{d}}{\text{d}z}\log\left(\frac{q}{q - \psi_X(z)}\right) =\frac{\psi_X'(z)}{q-\psi_X(z)} = \frac{\theta + z\sigma^2 + \mu\sqrt{\left(1 - \frac{z}{\rho}\right)\left(1 - \frac{z}{\hat{\rho}}\right)}}{\sqrt{\left(1 - \frac{z}{\rho}\right)\left(1 - \frac{z}{\hat{\rho}}\right)}\left(q - \frac{1}{\kappa} + \frac{1}{\kappa}\sqrt{\left(1 - \frac{z}{\rho}\right)\left(1 - \frac{z}{\hat{\rho}}\right)} - \mu z\right)},
\end{align}
which is a well defined function on $\mathbb{C}\backslash(\infty,\hat{\rho}]\cup[\rho,\infty)$ except possibly at the points $\zeta$ and $\hat{\zeta}$, which may be simple poles. Let us now proceed on a case-by-case basis.\\ \\
\underline{Case: I-A} \\ \\
In this case, the function $q - \psi_X(z)$ has no zeros. In particular, the singularities of $\Phi_q(z)/z$ in the open right half-plane are restricted to $\rho$. To derive a general expression for $G^+_q(t)$ we apply Formula \ref{eq:post} with some $0 < C < \rho$ and consider the integral of $e^{-zt}\Phi_q(z)/z$ along the line $C + \iota\mathbb{R}$ for fixed $t > 0$. To evaluate this, we consider instead the integral of our function along the contour $H$ of Figure \ref{fig:contour}, i.e.
\begin{align}\label{eq:cauchyint}
\int_{C - \iota R}^{C + \iota R}e^{-tz}\frac{\Phi_q(z)}{z}\text{d}z = -\int_{\eta_1}e^{-tz}\frac{\Phi_q(z)}{z}\text{d}z = \sum_{i=2}^{6}\int_{\eta_i}e^{-tz}\frac{\Phi_q(z)}{z}\text{d}z,
\end{align}
where we have used Cauchy's integral theorem in the last equality. In the limit $R \rightarrow \infty$ the integrals along the contours $\eta_2$ and $\eta_6$ vanish. Then, letting $\delta \rightarrow 0$ we have 
\begin{align*}
e^{-tu}\frac{\Phi_q(u)}{u} = \iota e^{-tu}\frac{\theta + u\sigma^2 -\iota \mu\sqrt{-\left(1 - \frac{u}{\rho}\right)\left(1 - \frac{u}{\hat{\rho}}\right)}}{u\sqrt{-\left(1 - \frac{u}{\rho}\right)\left(1 - \frac{u}{\hat{\rho}}\right)}\left(q - \frac{1}{\kappa} - \frac{\iota}{\kappa}\sqrt{-\left(1 - \frac{u}{\rho}\right)\left(1 - \frac{u}{\hat{\rho}}\right)} - \mu u\right)}, \quad u \in (\rho + \epsilon,\infty).
\end{align*}
along the contour $\eta_3$. Similarly along $\eta_5$ we have
\begin{align*}
e^{-tu}\frac{\Phi_q(u)}{u} = -\iota e^{-tu}\frac{\theta + u\sigma^2 +\iota \mu\sqrt{-\left(1 - \frac{u}{\rho}\right)\left(1 - \frac{u}{\hat{\rho}}\right)}}{u\sqrt{-\left(1 - \frac{u}{\rho}\right)\left(1 - \frac{u}{\hat{\rho}}\right)}\left(q - \frac{1}{\kappa} + \frac{\iota}{\kappa}\sqrt{-\left(1 - \frac{u}{\rho}\right)\left(1 - \frac{u}{\hat{\rho}}\right)} - \mu u\right)}, \quad u \in (\rho + \epsilon,\infty).
\end{align*}
Now, adding these integrands together and letting $\epsilon \rightarrow 0$ -- note that the integral along $\eta_4$ vanishes with $\epsilon$ -- we arrive at
\begin{align}\label{eq:fqright}
F_q(\infty) - G_q^+(t) = \frac{1}{2 \pi \iota}\int_{C+\iota\mathbb{R}}e^{-tz}\left(\frac{\Phi_q(z)}{z}\right)\text{d}z = \frac{a}{\pi}\int_{\rho}^{\infty}e^{-tu}\frac{bu - c}{u(u - \zeta)(u - \hat{\zeta})\sqrt{(u - \rho)(u - \hat{\rho})}}\text{d}u.
\end{align}
Using $\hat{\rho} < -C < 0$, the contour $\Gamma$ from Figure \ref{fig:contour}, and the identical approach we can show that for $t < 0$
\begin{align}\label{eq:fqleft}
G_q^-(t) = \frac{1}{2 \pi \iota}\int_{-C+\iota\mathbb{R}}e^{-tz}\left(-\frac{\Phi_q(z)}{z}\right)\text{d}z = -\frac{a}{\pi}\int^{\hat{\rho}}_{-\infty}e^{-tu}\frac{bu - c}{u(u - \zeta)(u - \hat{\zeta})\sqrt{(u - \rho)(u - \hat{\rho})}}\text{d}u.
\end{align}
Note that since we have assumed that $\rho \neq \zeta$ and $\hat{\rho} \neq \hat{\zeta}$, and since $\rho \neq \hat{\zeta}$ and $\hat{\rho} \neq \zeta$ by Proposition \ref{prop:themostannoying} (vi), the singularities in the integrals \eqref{eq:fqright} and \eqref{eq:fqleft} at $\rho$ and $\hat{\rho}$ remain integrable.\\ \\ 
\noindent \underline{II-A} \\ \\
\noindent For this case, the derivation of the function $G^-_q(t)$ remains the same. The major difference is that since $\zeta$ is a solution of \eqref{eq:main}, $\Phi_q(z)/z$ has a simple pole in the interval $(0,\rho)$ (see Proposition \ref{prop:themostannoying} (iii) and (iv)). The effect of this is that \eqref{eq:fqright} becomes 
\begin{align*}
F_q(\infty) - G_q^+(t) = \frac{a}{\pi}\int_{\rho}^{\infty}e^{-tu}\frac{bu - c}{u(u - \zeta)(u - \hat{\zeta})\sqrt{(u - \rho)(u - \hat{\rho})}}\d u + \frac{e^{-t\zeta}}{\zeta}.
\end{align*}
\underline{III-A} \\ \\
In this case the key difference is that both $\zeta$ and $\rho$ solve the associated quadratic equation \eqref{eq:assocquad}. It is easy to see that if this is the case and $\mu = 0$, then also $q = \frac{1}{\kappa}$ and therefore also $\hat{\rho} = \hat{\zeta}$ (see Proposition \eqref{prop:themostannoying} (vii)), which we explicitly assume is not the case here (we have chosen case A). Thus, we assume that $\mu \neq 0$, which together with our previous assumptions implies that $\rho = \left(q - \frac{1}{\kappa}\right)/\mu$. This has two consequences. The first is that \eqref{eq:fqleft} simplifies to
\begin{align*}
G_q^-(t) = -\frac{a}{\pi}\int^{\hat{\rho}}_{-\infty}e^{-tu}\frac{b}{u(u - \hat{\zeta})\sqrt{(u - \rho)(u - \hat{\rho})}}\text{d}u,
\end{align*}
and the second is that the integral along the contour $\eta_4$ does not vanish as $\epsilon \rightarrow 0$. Making the change of variables $z = \rho + \epsilon e^{\iota w}$, taking the limit as $\epsilon \rightarrow 0$, and applying the dominated convergence theorem shows that
\begin{align*}
\lim_{\epsilon \rightarrow 0}\frac{1}{2\pi \iota}\int_{\eta_4}e^{-tz}\frac{\Phi(z)}{z}\text{d}z = -\frac{1}{2 \pi \iota}\int_{0}^{2\pi}\lim_{\epsilon \rightarrow 0}e^{-t(\rho + \epsilon e^{\iota w})}\frac{\Phi(\rho + \epsilon e^{\iota w})}{\rho + \epsilon e^{\iota w}}\iota\epsilon e^{\iota w}\text{d}w = \frac{e^{-t\rho}}{2\rho}.
\end{align*}
Therefore we have
\begin{align*}
F_q(\infty) - G_q^+(t) = \frac{a}{\pi}\int_{\rho}^{\infty}e^{-tu}\frac{b}{u(u - \hat{\zeta})\sqrt{(u - \rho)(u - \hat{\rho})}}\text{d}u + \frac{e^{-t\rho}}{2\rho} ,
\end{align*}
where the same simplification takes place in the integrand as for $G_q^-(t)$.
\\ \\
Now, independent of the case, we remark that $G^-_q(t)$ and $G^+_q(t)$ are not only continuous, but also differentiable functions on $(-\infty,0)$  and  $(0,\infty)$ respectively. By (O6), this implies that we have identified an explicit expression for $F_q(t)$ at every point $t \neq 0$. However, since $F_q(t)$ is by assumption right-continuous, and since $F_q(0)- F_q(0-) = \int_{\{0\}}x\nu_q(\text{d}x) = 0$, we can actually conclude that $F_q(t)$ is continuous also at zero and define $F_q(0)$ as either $G^-_q(0-)$ or equivalently as $G^+_q(0+)$. If we define $g^-_q(t) := (G^-_q)'(t)$ for $t < 0$ and $g^+_q(t) := (G^+_q)'(t)$ for $t > 0$, the above discussion shows that 
\begin{align*}
F_q(t) = \int_{-\infty}^{t}\ind(x > 0)g_q^+(x)  + \ind(x < 0)g^-_q(x)\text{d}x, \quad t \in \mathbb{R}.
\end{align*}
Employing (O7) then gives the desired result.
\end{proof}
\begin{figure}[!htb]
\centering
\includegraphics[scale=0.60]{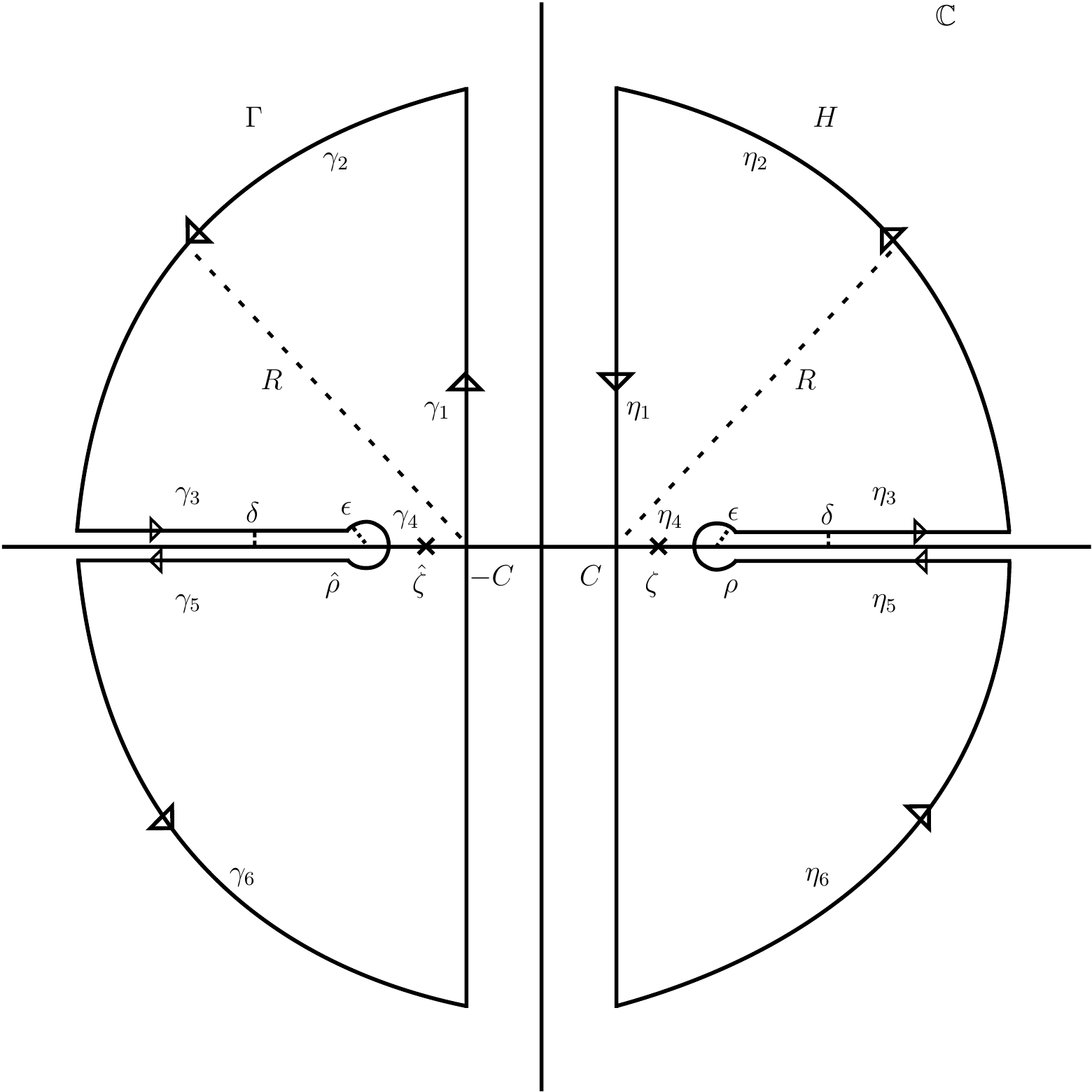}
\caption{Contours of integration}
\label{fig:contour}
\end{figure}
\noindent Going forward, for a measure $\mu$, let $_*{\mu}$ denote the pushforward measure under the transformation $x \mapsto -x$.
\begin{corollary}\label{cor:ggc}
If $\omega_q^+$ (resp. $-_*\omega^-_q$)  is a positive measure  then $S_{\eq}$ (resp. $-I_{\eq}$) has a distribution that is a GGC. The corresponding Thorin measure is given by $\tau_q^+ = \omega_q^+$ (resp. $\tau_q^- = -_*\omega_q^-$).
\end{corollary}
\begin{proof}
This follows from Theorem \ref{theo:main} and Theorem \ref{theo:bondy}.
\end{proof}
\begin{remark} From the definition of the measures $\omega_q^+$ and $-_*\omega^-_q$, it is clear they are not always positive (see Example \ref{ex:negmeas} below). That is, we may not conlude, in general, that the distributions of $S_{\eq}$ and of $-I_{\eq}$ are GGCs. Determining whether or not the measures are positive or signed is straightforward, we simply need to determine the slope and intercept of the line $\ell(u) := bu - c$ in \eqref{eq:themus} or the sign of $b$ in \eqref{eq:thenus} and \eqref{eq:thelams}. In doing so, it is easy to see that in the signed case the measure $\omega_q^+$ always breaks down into the difference of two finite, positive measures with the following characteristics: The measure that contributes positive mass is supported either on an interval, at one point, or on the union of an interval and a disjoint point. The measure that contributes negative mass will always be supported on a non-empty interval. Each measure assigns no mass to a non-empty interval $(0,R)$, where either $R= \zeta$ or $R = \rho$. This breakdown describes the Jordan decomposition of the measure $\omega_q^+$, which we can always determine exactly in this manner.  Further, the above statements are equally true for the measure $-_*\omega^-_q$, with $R = -\hat{\zeta}$ or $R = -\hat{\rho}$.\demo
\end{remark}
\begin{example}\label{ex:negmeas}
We consider an example to demostrate an instance of the previous remark. Let
\begin{align*}
(\theta,\mu,\kappa,\sigma,q) = \left(-1,\frac{7}{32},16,1,\frac{19}{64}\right).
\end{align*}
In this case $b$, the slope of the line $\ell(u)$, is positive, and the intercept with the horizontal axis, $c/b = 127/8 = 15.875$, occurs to the right of $\rho = 1 + \sqrt{17}/4 \approx 2.030776$. Additionally, while $\zeta \approx 1.805903$ solves \eqref{eq:main} $\hat{\zeta} \approx 0.256043$ does not, i.e. we have described an instance of Case II-A. Thus we have
\begin{align*}
\omega_q^+(\text{d}u) = &\underbrace{\delta_\zeta(\text{d}u) \,+ \, \ind\left( \frac{c}{b} < u < \infty \right)\frac{a(bu - c)}{\pi(u - \zeta)(u - \hat{\zeta})\sqrt{(u - \rho)(u - \hat{\rho})}}\text{d}u}_{+} \\ & - \,\underbrace{\ind\left( \rho < u < \frac{c}{b}\right)\frac{a(c - bu)}{\pi(u - \zeta)(u - \hat{\zeta})\sqrt{(u - \rho)(u - \hat{\rho})}}\text{d}u}_{-}.
\end{align*} 
where $+$ and $-$ denote the positive and negative contributions of $\omega_q^+$ respectively.\rdemo
\end{example}
\noindent For the following Corollary, we will require the ideas from the previous discussion as well as  the Frullani identity, which states that for a continuously differentiable function $f(x)$ we have
\begin{align*}
\int_{0}^{\infty}\frac{f(ax) - f(bx)}{x}\text{d}x = (f(0) - f(\infty))\log\left(\frac{b}{a}\right),
\end{align*} 
where we assume $0 \leq a \leq b$ and that $f(0)$ and $f(\infty)$ are finite. 
\begin{corollary}\label{cor:gegc}
The Laplace exponent of $S_{\eq}$ (resp. $-I_{\eq}$) has the form
\begin{align}\label{eq:ggclike}
\psi_q^+(z) = \int_{\mathbb{R}^+}\log\left(\frac{u}{u-z}\right)\tau_q^+(\text{d}u),\quad \left(\text{resp.  } \psi_q^-(-z) = \int_{\mathbb{R}^+}\log\left(\frac{u}{u-z}\right)\tau_q^-(\text{d}u) \right)
\end{align}
where $\tau_q^+ = \omega_q^+$ (resp. $\tau_q^- = -_*\omega_q^-$). The equality \eqref{eq:ggclike} holds for $\re(z) < R$, where $R = \zeta$ (resp. $R = -\hat{\zeta}$) whenever $\zeta$ (resp. $\hat{\zeta}$) satisfies \eqref{eq:main}, and $R = \rho$ (resp. $R = -\hat{\rho}$) otherwise.
\end{corollary}
\begin{proof}
If $\omega^+_q$ is a postive measure, then the result follows immediately from Corollary \ref{cor:ggc}. Otherwise, let $\chi^{\pm}$ denote the Jordan decomposition of $\omega^+_q$, which has the relatively simple form described in the discussion preceding the statement of the corollary, and assume first that $z \leq 0$. Then
\begin{align*}
\int_{\orp}\left\vert \log\left(\frac{u}{u-z}\right)\right\vert\chi^{+}(\text{d}u) < \infty,\quad\text{ and }\quad \int_{\orp}\left\vert \log\left(\frac{u}{u-z}\right)\right\vert\chi^{-}(\text{d}u) < \infty,
\end{align*}
since each of $\chi^{+}$ and $\chi^{-}$ is finite and assigns no mass to the interval $(0,R)$. Applying Frullani's identity with $f(x) = -e^{-x}$, $a = y$, $b = y - z$ we get
\begin{align*}
\int_{\orp}\left\vert \log\left(\frac{u}{u-z}\right)\right\vert\chi^{+}(\text{d}u) &= \int_{\orp}\left\vert\int_0^{\infty}\frac{e^{-(u-z)x} - e^{-ux}}{x}\text{d}x\right\vert \chi^{+}(\text{d}u) \\ &= \int_{\orp}\int_{0}^{\infty}\left\vert (e^{zx} - 1) \frac{e^{-ux}}{x} \right\vert\text{d}x\chi^{+}(\text{d}u) < \infty,
\end{align*}
with the identical result for $\chi^-$. It follows that we may apply Fubini's Theorem for each of $\chi^+$ and $\chi^-$ separately, which, after recombining, establishes the result for $z \leq 0$. For $0 < z < R$ repeating the above excercise with $f(x) = e^{-x}$, $a = y - z$, $b = y$ shows that the result can be extend to $z < R$.  However, it is not difficult to see that
\begin{align*}
\int_{\mathbb{R}^+}\log\left(\frac{u}{u-z}\right)\omega_q^+(\text{d}u),\quad\text{ and }\quad \int_{\orp}\left(e^{zx} - 1\right)\pi^+_q(x)\text{d}x,
\end{align*}
are analytic functions for $\re(z) < R$. By analytic continuation, the functions must be equal on this half-plane. The proof for $-_*\omega^-_q$ is identical.
\end{proof}
\noindent We can also use the given results to determine the distribution of the  overall supremum $S_{\infty} := \lim_{t\rightarrow \infty}S_t \overset{a.s.}{=} \lim_{q\rightarrow0}S_{\eq}$ and overall infimum $I_{\infty} := \lim_{t\rightarrow \infty}I_t \overset{a.s.}{=} \lim_{q\rightarrow0}I_{\eq}$, which exist as real valued random variables when $\e[X_1] = \theta + \mu < 0$ and  $0 < \theta + \mu$ respectively. To do so, we need to consider the limits $\lim_{q \rightarrow 0}\psi^{+}_q(z)$ and $\lim_{q \rightarrow 0}\psi^{-}_q(z)$. \\ \\
\noindent In what follows, we allow $\zeta$ and $\hat{\zeta}$ to extend to the case $q = 0$. It is easy to show that for this case we have $\{\zeta(0),\hat{\zeta}(0)\} = \{0,-2(\theta + \mu)/(\kappa\mu^2 + \sigma^2)\}$, where the assignment of the zero root to either $\zeta(0)$ or $\hat{\zeta}(0)$ depends on the value of $\theta + \mu$. Likewise, the terms $b$ and $c$ from \eqref{eq:const} and therefore also the measures $\mu_q^{\pm}$, $\nu_q^{\pm}$ and $\lambda_q^{\pm}$ from \eqref{eq:themus}, \eqref{eq:thenus}, and \eqref{eq:thelams} respectively are all well-defined also for $q = 0$.
\begin{corollary}\label{cor:q0}
If $\theta + \mu < 0$ (resp. $\theta + \mu > 0$) then the Laplace exponent of $S_{\infty}$ (resp. $-I_{\infty}$) has the form 
\begin{align}\label{eq:limcaseii}
\psi_0^+(z) = \int_{\orp}\log\left(\frac{u}{u-z}\right)\tau_0^+(\text{d}u)\quad\quad\left(\text{ resp. }\psi_0^-(-z) = \int_{\orp}\log\left(\frac{u}{u-z}\right)\tau_0^-(\text{d}u)\right),
\end{align}
where $\tau^+_0 = \omega^+_0$ (resp. $\tau_0^- = -_*\omega^-_0$) and 
\begin{align*}
\omega_0^{+} &= 
\begin{dcases}
\mu_0^+ &,  \quad \zeta(0) \text{ is not a solution of } \psi_X(z) = 0 \text{ and } \rho \neq \zeta(0)\\
\mu_0^+ + \delta_{\zeta(0) } &,  \quad \zeta(0)  \text{ is a solution of } \psi_X(z) = 0 \text{ and }  \rho \neq \zeta(0)\\
\nu_0^+ + \delta_{\zeta(0) } &, \quad \rho=\zeta(0)
\end{dcases}\\
\left(\text{resp.  }\vphantom{\begin{dcases}
\mu_0^- &,  \quad \hat{\zeta}(0) \text{ is not a solution of } \psi_X(z) = 0 \text{ and } \hat{\rho} \neq \hat{\zeta}(0)\\
\mu_0^- - \delta_{\hat{\zeta}(0) } &,  \quad \zeta(0)  \text{ is a solution of } \psi_X(z) = 0 \text{ and }  \hat{\rho} \neq \hat{\zeta}(0)\\
\nu_0^- - \delta_{\hat{\zeta}(0) } &, \quad \hat{\rho}=\hat{\zeta}(0)
\end{dcases}} \omega_0^{-} \right.&= \left.
\begin{dcases}
\mu_0^- &,  \quad \hat{\zeta}(0) \text{ is not a solution of } \psi_X(z) = 0 \text{ and } \hat{\rho} \neq \hat{\zeta}(0)\\
\mu_0^- - \delta_{\hat{\zeta}(0) } &,  \quad \zeta(0)  \text{ is a solution of } \psi_X(z) = 0 \text{ and }  \hat{\rho} \neq \hat{\zeta}(0)\\
\lambda_0^- - \delta_{\hat{\zeta}(0) } &, \quad \hat{\rho}=\hat{\zeta}(0)
\end{dcases}
\right).
\end{align*}
The equality \eqref{eq:limcaseii} holds for $\re(z) < R$, where $R = \zeta(0)$ (resp. $R = -\hat{\zeta}(0)$) whenever $\zeta(0)$ (resp. $\hat{\zeta}(0)$) satisfies $\psi_X(z) = 0$ and $R = \rho$ (resp. $R = -\hat{\rho}$) otherwise.
\end{corollary}
\begin{proof}
We will work through the three possible cases for $S_{\infty}$; the derivation for $-I_{\infty}$ is identical. First, let us make five observations -- essentially extensions of Proposition \ref{prop:themostannoying} for the case $q=0$ plus two obvious facts -- that are easy to verify: (a) $\zeta$ and $\hat{\zeta}$ are real for $q$ small enough; (b) neither $\rho = \hat{\zeta}(0)$ nor $\hat{\rho} =\zeta(0)$ is possible; (c) $z_0$ is a solution of $\psi_X(z) = 0$ iff $z_0 = \hat{\zeta}(0)$ or $z_0 = \zeta(0)$ and $-1/\kappa \leq \mu z_0$; (d) the equation $\rho = \zeta(q)$ (resp. $\hat{\rho} = \hat{\zeta}(q)$) has at most two solutions; and (e) the assumption $\theta + \mu < 0$ implies that $\zeta(0) = -2(\theta + \mu)/(\kappa\mu^2 + \sigma^2) > 0$ and, in particular since $\hat{\rho} < 0$, that $\hat{\rho} \neq \hat{\zeta}=0$.\\ \\
\underline{Case 1}\\ \\
\noindent We assume first that $\zeta(0) $ does not solve $\psi_X(z) = 0$ such that $ \rho\neq\zeta(0) $. It is easy to see that the first part of our assumption, together with observations (a) and (c), implies that $q - 1/\kappa > \mu\zeta$ for $q$ small enough. Applying Proposition \ref{prop:themostannoying} (ii), we see that $\zeta$ is not a solution of \eqref{eq:main} when $q$ is small. Further, from observation (d) it is clear that we may also assume that neither $\rho = \zeta$ nor $\hat{\rho}=\hat{\zeta}$ and therefore that $\omega_q^+ = \mu_q^+$ for small $q$. Then, since: a) the function $\left\vert\log\left(\frac{u}{u-z}\right)\right\vert$ is bounded for $u \in (\rho,\infty)$ for every fixed $z$ such that $\re(z) < \rho$; b) the measure 
\begin{align*}
\frac{\ind(u > \rho)}{\sqrt{u}\sqrt{(u-\rho)(u-\hat{\rho})}}\text{d}u
\end{align*}
is finite; and c) the function
\begin{align}\label{eq:funq}
\left\vert\frac{(bu - c)\sqrt{u}}{(u - \zeta(q))(u - \hat{\zeta}(q))}\right\vert
\end{align}
is bounded for $ (u,q)\in (\rho,\infty)\times[0,\epsilon]$ for $\epsilon$ small enough (due to our assumption that $\rho \neq \zeta(0)$, observations (a) and (b), and Proposition \ref{prop:themostannoying} (v) we know that $\zeta$ and $\hat{\zeta}$ are both strictly less than $\rho$), we can apply the dominated convergence theorem in the integral in \eqref{eq:ggclike} to get the result. Note that observations (d) and (e) ensure that we do not have any cancellation in the numerator and denominator in \eqref{eq:funq} as $q\rightarrow0$, i.e. it is not possible that $\mu_q^+$ becomes $\lambda_q^+$ in the limit.\\ \\
\underline{Case 2}\\ \\
\noindent If we assume that $\zeta(0) $ solves $\psi_X(z) = 0$ and $ \rho\neq\zeta(0) $, then the approach is essentially the same, except that we must show that $\omega_q^+ = \mu_q^+ + \delta_{\zeta}$ for $q$ small enough, i.e. that $\zeta$ becomes a solution of \eqref{eq:main} for $q$ small enough. Proposition \ref{prop:themostannoying} (ii) together with the fact that our assumptions imply that 
\begin{align}\label{eq:thetest}
-1/\kappa < \mu\zeta(0),
\end{align}
shows that $\zeta$ does solve \eqref{eq:main} for small $q$.  To verify \eqref{eq:thetest} we recall that observation (e) states that $\zeta(0) = -2(\theta + \mu)/(\kappa\mu^2 + \sigma^2) > 0$ and consider cases for $\mu$. If $\mu\geq 0$, then clearly \eqref{eq:thetest} clearly holds. If instead $\mu < 0$ and $-1/\kappa > \mu\zeta(0)$, then $\zeta(0)$ is not a solution of $\psi_X(z) = 0$ according to general observation (c), which contradicts our assumptions. Finally, if $\mu < 0$ and $-1/\kappa = \mu\zeta(0)$, then solving for $\sigma$ yields $\sigma = \sqrt{\kappa\mu(2\theta + \mu)}$. Plugging this into the expression for $\rho$ yields $\rho = -1/(\mu\kappa) = \zeta(0)$, which again contradicts our assumptions.\\ \\
\underline{Case 3}\\ \\
\noindent Assuming now that $\rho = \zeta(0)$, we solve this equation for $\sigma^2$, which yields $\sigma^2 = \kappa\mu(2\theta + \mu)$. Plugging this value of $\sigma^2$ into the expression for $\rho$, shows that $\rho = \zeta(0) = -1/(\mu\kappa)$ from which it follows that $\mu  <  0$. General observations (a) and (d) and Proposition \ref{prop:themostannoying} (v), show that $\zeta < \rho = -1/(\kappa\mu)$ for $q$ small enough. It follows that $\mu\zeta > q - 1/\kappa$ for $q$ small enough, which, according to Proposition \ref{prop:themostannoying} (ii), implies that $\zeta$ solves \eqref{eq:main} for $q$ small enough. Therefore, $\omega_q^+ = \mu_q^+ + \delta_\zeta$ when $q$ is small. \\ \\
\noindent To complete the proof for this case, we need to following facts, which have straightforward proofs that are therefore omitted (although (c) requires some rather tedious algebra): (a) $c/b > \rho$ for $q$ small enough; (b) $\lim_{q\rightarrow0}c/b = \rho$ such that for every $u \in (\rho,\infty)$ we have
\begin{align*}
\lim_{q\rightarrow 0}\frac{u - \frac{c}{b}}{(u - \zeta)(u-\hat{\zeta})} = \frac{1}{u};
\end{align*}
and (c) $\lim_{q\rightarrow 0}(\rho - c/b)/\sqrt{\rho - \zeta} = 0$. We now aim to show that
\begin{align}\label{eq:intsetup}
\lim_{q\rightarrow 0}\int_{\rho}^{\infty}\left\vert\frac{u-\frac{c}{b}}{(u-\zeta)(u-\hat{\zeta})\sqrt{(u-\rho)(u-\hat{\rho})}}\right\vert\text{d}u = \int_{\rho}^{\infty}\frac{1}{u\sqrt{(u-\rho)(u-\hat{\rho})}}\text{d}u,
\end{align}
as this, together with the already mentioned fact that the function $\left\vert\log\left(\frac{u}{u-z}\right)\right\vert$ is bounded for $u \in (\rho,\infty)$ for every fixed $z$ such that $\re(z) < \rho$, would allow us to use the generalized form of the dominated convergence theorem (see Theorem 19 pg. 89 in \cite{roy}) in the integral \eqref{eq:ggclike} and complete the proof for this case.\\ \\
\noindent The integral on the right of \eqref{eq:intsetup} is easily evaluated (see \eqref{eq:theintapp}):
\begin{align}\label{eq:whatwewant}
\int_{\rho}^{\infty}\frac{1}{u\sqrt{(u-\rho)(u-\hat{\rho})}}\text{d}u = \frac{1}{\sqrt{-\rho\hat{\rho}}}\left(\arctan\left(\frac{-(\rho + \hat{\rho})}{2\sqrt{-\rho\hat{\rho}}}\right) + \frac{\pi}{2}\right).
\end{align}
If we ignore the absolute value in the integral on the left of \eqref{eq:intsetup} for the moment, and treat it as an indefinite integral, we can evaluate the resulting integral exactly -- after a partial fraction decomposition and the substitutions $w = u - \zeta$ and $w = u - \hat{\zeta}$ -- by using the same techniques as for the integral on the right (constant of integration omitted):
\begin{align}\label{eq:thedeets}
I_0 := \int&\frac{u-\frac{c}{b}}{(u-\zeta)(u-\hat{\zeta})\sqrt{(u-\rho)(u-\hat{\rho})}}\text{d}u \\
&= \frac{\zeta - \frac{c}{b}}{\zeta - \hat{\zeta}}\underbrace{\int\frac{1}{w\sqrt{(w-(\rho - \zeta))(w-(\hat{\rho}-\zeta))}}\text{d}w}_{:=I_1} + \frac{\hat{\zeta}-\frac{c}{b}}{\hat{\zeta}-\zeta}\underbrace{\int\frac{1}{w\sqrt{(w-(\rho-\hat{\zeta}))(w-(\hat{\rho}-\hat{\zeta}))}}\text{d}w}_{:=I_2}\nonumber\\
&= \frac{\zeta - \frac{c}{b}}{\zeta - \hat{\zeta}}\frac{1}{\sqrt{-(\rho - \zeta)(\hat{\rho} - \zeta)}}\left(\arctan\left(\frac{2(\rho-\zeta)(\hat{\rho}-\zeta)-(\rho + \hat{\rho}-2\zeta)w}{2\sqrt{-(\rho - \zeta)(\hat{\rho} - \zeta)}\sqrt{(w-(\rho - \zeta))(w-(\hat{\rho}-\zeta))}}\right)\right)\nonumber\\
&\hphantom{=}+\frac{\hat{\zeta} - \frac{c}{b}}{\hat{\zeta}-\zeta}\frac{1}{\sqrt{-(\rho - \hat{\zeta})(\hat{\rho} - \hat{\zeta})}}\left(\arctan\left(\frac{2(\rho-\hat{\zeta})(\hat{\rho}-\hat{\zeta})-(\rho + \hat{\rho}-2\hat{\zeta})w}{2\sqrt{-(\rho - \hat{\zeta})(\hat{\rho} - \hat{\zeta})}\sqrt{(w-(\rho - \hat{\zeta}))(w-(\hat{\rho}-\hat{\zeta}))}}\right)\right)\nonumber. 
\end{align}
In order to evaluate the integral on the left-hand side of \eqref{eq:intsetup} we need to evalaute $I_0$ over the interval $(\rho,c/b)$ where the integrand is negative and then over the interval $(c/b,\infty)$ where the integrand is positive. We see, however, that irrespective of the limits of integration for $I_1$, the contribution from this term will vanish as $q\rightarrow 0$ since the arctangent function is bounded and since $(\zeta - c/b)/\sqrt{\rho - \zeta}$ goes to zero with $q$ (fact (c) from above). Thus, we need only consider the integral $I_2$ over these intervals. It is easy to see that $I_2$ evaluated over $(\rho - \hat{\zeta},c/b - \hat{\zeta})$ (i.e. $(\rho,c/b)$ expressed using transformed variable $w = u - \hat{\zeta}$) will vanish with $q$, since since $c/b \rightarrow \rho$ as $q \rightarrow 0$. Using this fact again, and also the fact that $\hat{\zeta}\rightarrow 0$ as $q \rightarrow 0$, the value of $I_2$ over the interval $(c/b - \hat{\zeta},\infty)$ converges to \eqref{eq:whatwewant} as $q$ goes to 0, and so we have proven \eqref{eq:intsetup}.
\end{proof}
\begin{remark}
Note that since both the class of GGCs and the class of MEs are closed with respect to weak convergence (see Proposition 9.10 and Corollary 9.6 together with Thereom A.4 in \cite{SSV2012}), the distributions of both $S_{\infty}$ and $-I_{\infty}$ will be MEs, and they will also be GGCs if the measures $\tau_q^+$ and $\tau_q^-$ are positive for $q$ small enough. This means that in the remainder of this paper, we can treat the case $q=0$ in exactly the same way that we would treat the case $q > 0$. Thus, unless otherwise stated, the reader may assume that the notation $S_{\eq}$ and $I_{\eq}$ includes the case $q=0$, i.e. $S_{\infty}$ and $I_{\infty}$.\demo
\end{remark}
\noindent We have shown that $S_{\eq}$ and $-I_{\eq}$ always have a Laplace exponent of the 
\begin{align}\label{eq:genrep}
\psi(z) = \int_{\orp}\log\left(\frac{u}{u-z}\right)\tau(\text{d}u),
\end{align}
where $\tau$ is the signed measure from Corollaries \ref{cor:ggc}, \ref{cor:gegc}, and \ref{cor:q0}
, which is derived from some linear combination of the measures \eqref{eq:themus}, \eqref{eq:thenus}, and \eqref{eq:thelams}, and the Dirac delta measure. At this point we have strong evidence that the distributions of $S_{\eq}$ and $-I_{\eq}$ are not GGCs whenever $\tau$ is not a positive measure. The following final corollary for this section confirms this assumption.
\begin{corollary}\label{cor:notggc}
Let $Y$ denote either $S_{\eq}$ or $-I_{\eq}$ with Laplace exponent $\psi(z)$ and associated measure $\tau$ as described in \eqref{eq:genrep}. If $\tau$ is not a positive measure, then the distribution of $Y$ is not a GGC.
\end{corollary}
\begin{proof}
We assume that the distribution of $Y$ is a GGC and that $\tau$ is not a positive measure. Therefore, $\tau$ must be a finite, signed measure such that there exist $0 < a < b < \infty$ for which $\int_{(a,b)}\tau(\text{d}u)<0$ and $\tau(\{a\}) = \tau(\{b\}) = 0$. Now we apply Lemma \ref{lem:ggcstiel}, which guarantees that $\psi'(z)$ has an analytic continuation of the form
\begin{align*}
\psi'(z) = \int_{\orp}\frac{\mu(\text{d}u)}{u-z},\quad z \in \mathbb{C}\backslash\crp
\end{align*}
for some positive measure $\mu$. The measure $\mu$ is uniquely determined by the function $\psi'(z)$ (by virtue of the fact that $\psi'(z)$ is a Pick function; see discussion top of pg. 30 and Theorem 2.4.1 in \cite{Bondesson}). In particular,
\begin{align}\label{eq:limmeas}
\lim_{y \downarrow 0} \frac{1}{\pi}\int_a^b\im\left(\psi'(x + \iota y)\right)\text{d}x = \int_{(a,b)}\mu(\text{d}u).
\end{align}
Expanding the left-hand side of \eqref{eq:limmeas} we get
\begin{align}\label{eq:bitingbond}
\frac{1}{\pi}\int_a^b\im\left(\psi'(x + \iota y)\right)\text{d}x  &= \frac{1}{\pi}\int_{a}^b\int_{\orp}\frac{y}{(u-x)^2 + y^2}\tau(\text{d}u)\text{d}x \\ &= \int_{\orp}\frac{1}{\pi}\left(\arctan\left(\frac{b-u}{y}\right) - \arctan\left(\frac{a-u}{y}\right)\right)\tau(\text{d}u),
\end{align}
where the interchange in the order of integration in the second equality in \eqref{eq:bitingbond} is justified by the fact that $\frac{y}{(u-x)^2 + y^2}$ is a bounded, positive function for each $y > 0$ and an argument identical to the one used in the proof of Corollary \ref{cor:gegc}. Now, the integrand on the right-hand side of \eqref{eq:bitingbond} is bounded by one, and, in fact, converges to one as $y$ approaches zero for $a < u < b$. For $u \notin [a,b]$ the integrand converges to zero. Applying the Dominated Convergence Theorem then shows that 
\begin{align*}
\int_{(a,b)}\mu(\text{d}u) = \int_{(a,b)}\tau(\text{d}u) < 0,
\end{align*}
which is a contradiction, since $\mu$ is a positive measure.
\end{proof}
\begin{remark}
Since $\text{ME} \not\subset \text{GGC}$ (and also $\text{GGC} \not\subset \text{ME}$) the result of Corollary \ref{cor:notggc} is not really surprising. However, the results of this section raise a potentially interesting avenue of further research, namely to attempt to define the class of probability distributions whose CGF has the form \eqref{eq:genrep}. The potentially difficult part of this exercise, is to settle on the proper definition of the ``Thorin" measure for this class, as it is easy to leave the realm of viable CGFs by a poor choice of signed measure. Additionally, while the literature on finite signed or complex measures is well developed, the literature on measures with infinite total variation is somewhat more limited, indicating the fact that working with such measures is more difficult. Ideally, we would like our class of distributions to include the class of GGCs, which would require at least some of the measures to have infinite total mass. The result of Corollary \ref{cor:notggc} is also interesting in the sense that although the NIG distribution is an extended generalized gamma convolution (EGGC), essentially a GGC extended to the real line (see Chapter 7 in \cite{Bondesson}), the Wiener-Hopf factors of the NIG process are not generally MGFs of GGCs.
\end{remark}
\section{Technical Details of the Approximation Algorithm}\label{sect:techdets}
In what follows, let $X$ denote an NIG process and $Y$ denote either $S_{\eq}$ or $-I_{\eq}$. Further let $\varphi(z) = \e[e^{zY}]$ and $\psi(z) = \log(\varphi(z))$. We have seen (Corollaries \ref{cor:gegc} and \ref{cor:q0}) that $\psi(z)$ has the form
\begin{align}
\psi(z) = \int_{\orp}\log\left(\frac{x}{x-z}\right)\tau(\text{d}x), \quad \re(z) \leq R,
\end{align}
where $\tau$ is a finite, possibly signed measure on $\orp$, which assigns no mass to a non-empty interval $(0,R)$. From Proposition \ref{prop:ggc} and Corollary \ref{cor:ggc} we know that if $\tau$ is a positive measure, then the $[n-1/n]$ Pad\'{e} approximant of $\psi'(z)$ can be used to construct a function, which is the MGF of an $n$-fold convolution of gamma distributions and matches the first $2n-1$ moments of the distribution of $Y$. Further we know from Theorem \ref{theo:rog}, Proposition \ref{prop:me}, and the fact that $X \in \mathcal{CM}$ that regardless of whether or not $\tau$ is positive, the $[n-1/n]$ Pad\'{e} approximant of $\varphi(z)$ is the MGF of a finite mixture of exponentials, which also matches the first $2n-1$ moments of the distribution of $Y$. Thus we have potentially two approaches for approximation, which depend on the Taylor series expansion of either $\psi'(z)$ or of $\varphi(z)$. \\ \\
\noindent Regardless of whether or not $\tau$ is positive, we may readily show that $\psi(z)$ is analytic near zero and that we may repeatedly differentiate under the integral sign, such that
\begin{align*}
\psi^{(k)}(z) = (k-1)!\int_{[R,\infty)}\frac{\tau(\text{d}x)}{(x - z)^k},\quad k \geq 1,
\end{align*} 
for $z$ near zero. Thus, $\psi'(z)$ has the following Taylor series expansion at zero
\begin{align*}
\psi'(z) = \sum_{k=0}^{\infty}m_{k+1}z^k, \quad \text{where}\quad m_k := \frac{\psi^{(k)}(0)}{(k-1)!} = \int_{[R,\infty)}x^{-k}\tau(\text{d}x).
\end{align*}
We see that $\{m_k\}_{k \geq 1}$ are simply the negative moments of the measure $\tau$ and that these are related to the cumulants $\{\kappa_k\}_{k\geq 1}$ of the distribution of $Y$ by the relation $\kappa_k = (k-1)!m_k$. If $\{\mu_k\}_{k\geq1}$ are the moments of the distribution of $Y$, then we also have
\begin{align*}
\varphi(z) = \sum_{k=0}^{\infty}\frac{\mu_k}{k!}z^k, \quad \text{where}\quad \mu_0 = 1,\,\mu_1 = m_1,\,\text{and }\,\mu_k = (k-1)!m_k + \sum_{j=1}^{k-1}\binom{k-1}{j-1}(j-1)!m_j\mu_{k-j},
\end{align*}
for $k \geq 2$, which follows from the well known relationship between moments and cumulants. We see that our approximation depends only on our ability to compute the negative moments of $\tau$.
\subsection{Computing the negative moments of $\tau$}
Conveniently, we can compute the negative moments of $\tau$ exactly, i.e. without resorting to numerical integration. Recalling that $\tau$ is a stand-in for the measures $\omega_q+$, and $-_*\omega_q^-$ and consulting Table \ref{tab:meas} along with Formulas \ref{eq:themus} through \ref{eq:thelams}, we observe that the challenging part of computing the negative moments of $\tau$ is computing an integral whose general form is
\begin{align*}
I_{i,\,j,\,k} := \int_{R}^{\infty}\frac{Ax + B}{x^k(x-D)^{i}(x-E)^{j}\sqrt{(x-C)(x-R)}}\text{d}x,
\end{align*}
where $k \in \mathbb{N}$, $i,\,j \in \{0,1\}$, $A,\,B,\,C,\,R \in \mathbb{R}$ such that: $Ax+B \not\equiv 0$, $C < 0 < R$, and $C < D \leq E < R$ whenever $D$ and $E$ are both real. In particular, for $D\,,E \in \mathbb{R}$ we have $(C - D)(R - D) < 0$, and $(C-E)(R-E)<0$. Further, $D$ and $E$ are either both real or both have nonzero imaginary part; in the latter case we have $D = \bar{E}$. \\ \\
\noindent The approach to computing $I_{i,\,j,\,k}$ is simply to recognize that
\begin{align*}
I_{i,\,j,\,k} = \int_{R}^{\infty}\frac{W(x)}{\sqrt{P(x)}}\text{d}x,
\end{align*}
where $W(x)$ is a rational function and $P(x) := (x-C)(x - R)$. It is always possible to reduce this rational function via partial fraction decomposition into a sum of integrals of the form
\begin{align}\label{eq:easyint}
\int_{R}^{\infty}\frac{K}{(x-J)\sqrt{P(x)}}\text{d}x,
\end{align}
for constants $K$ and $J \in \{0,\,C,\,D\}$. If $J$ is real, then the integrals \eqref{eq:easyint} can be computed exactly via the following identities (see Formulas 2.266, 2.268, 2.269.1-2 in \cite{Jeffrey2007}), where $Q(x) = a + bx + cx^2$, $a < 0$, and the integrals are intended in the indefinite sense; the constant of integration is omitted:
\begin{align}\label{eq:theintapp}
L_1 &= \frac{1}{\sqrt{-a}}\arctan\left(\frac{2a + bx}{2\sqrt{-a}\sqrt{Q(x)}}\right),\nonumber\\
L_2 &= -\frac{\sqrt{Q(x)}}{ax} - \frac{b}{2a}J_1,\\
L_k &= -\frac{\sqrt{Q(x)}}{(k-1)ax^{k-1}} - \frac{(2k-3)b}{2(k-1)a}J_{k-1} - \frac{(k-2)c}{(k-1)a}J_{k-2},\quad k \geq 3, \nonumber
\end{align}
where
\begin{align*}
L_{k} := \int\frac{1}{x^k\sqrt{Q(x)}}\text{d}x,\quad k\in \mathbb{N}. 
\end{align*}
If $J$ is not real, then we require a different approach, which we demostrate in the following example, in which we show how to compute the most challenging version of $I_{i,\,j,\,k}$.
\begin{example}
Let
\begin{align*}
I_k := \int_{R}^{\infty}\frac{Ax + B}{x^k(x-D)(x-E)\sqrt{(x-C)(x-R)}}\text{d}x,
\end{align*}
and expand the rational portion of the integrand as a partial fraction, such that
\begin{align}\label{eq:intbas}
I_k = \sum_{j=1}^k\int_R^{\infty}\frac{a_{1,k-j}}{x^j\sqrt{P(x)}}\text{d}x + \int_{R}^{\infty}\frac{a_{2,1}x + a_{2,0}}{(x -D)(x-E)\sqrt{P(x)}}\text{d}x,
\end{align}
where
\begin{gather*}
a_{2,0} = a_{1,k-1}(D+E) - a_{1,k-2},\quad a_{2,1} = -a_{1,k-1},\\
a_{1,j} = \frac{a_{1,j-1}(D+E) - a_{1,j-2}}{DE},\quad 2 \leq j \leq k-1, \quad\text{ and }\\
a_{1,1} = \frac{A}{DE} + \frac{B(D+E)}{(DE)^2}\quad a_{1,0} = \frac{B}{DE}.
\end{gather*}
Note that the above recursion can also be solved explicitly, but for computational purposes the recursion will be faster. With this decomposition, we recognize that the integrals in the first term on the right-hand side of \eqref{eq:intbas} can be calculated using \eqref{eq:theintapp} directly. The method of computing the second integral on the right-hand side of \eqref{eq:intbas} depends on the values of $D$ and $E$. If $D$ and $E$ are real and $D\neq E$, then we must do one more partial fraction expansion in the integral on the right in \eqref{eq:intbas}. The substitutions $x \mapsto x + D$ and $x \mapsto x + E$ in the resulting integrals, combined with the fact that $P(D) = (C-D)(R-D) < 0$ and $P(E) = (C-E)(R-E) < 0$ and \eqref{eq:theintapp} allow for exact evaluation of $I_k$. Note that if $D=E$, then we can omit the partial fraction decomposition and use \eqref{eq:theintapp} directly. \\ \\
\noindent If $D$ and $E$ are complex we proceed analogously with one further partial fraction expansion plus one additional step, namely the Euler substitution $P(x) = x + t$. This transforms the second integral in \eqref{eq:intbas} into two integrals with rational integrands, in particular
\begin{align}\label{eq:finalint}
\int_{R}^{\infty}\frac{1}{(x-D)\sqrt{P(x)}}\text{d}x = \int_{-R}^{-\frac{C+R}{2}}\frac{2}{t^2 + 2Dt + C(D-R) + DR}\text{d}t = \int_{-R}^{-\frac{C+R}{2}}\frac{2}{(t - r^+)(t-r^-)}\text{d}t,
\end{align}
where
\begin{align*}
r^{\pm} := -D \pm \sqrt{(D-C)(D-R)}.
\end{align*}
By expressing $C$ and $D$ in terms of $\rho$ and $\hat{\rho}$ we can show that the interval $(-R,-\frac{1}{2}(C+R))$ is never empty, and from our assumptions that $D \in \mathbb{C}\backslash\mathbb{R}$, we can show that the function $(t-r^+)(t-r^-)$ has no roots in $\mathbb{R}$. Therefore the integral on the right of \eqref{eq:finalint} is easily evaluated exactly as
\begin{align*}
\frac{2}{r^+ - r^-}\left(\log\left(-\frac{C+R}{2} - r^+\right) + \log\left(-R - r^-\right) - \left(\log\left(-\frac{C+R}{2} - r^-\right) + \log\left(-R - r^+\right)\right)\right)
\end{align*}
Of course, the same approach works also for $E$.\rdemo
\end{example}
\section{Examples and Applications}\label{sect:exandapp}
In this section $X$ refers to a NIG process as do the random variables $S_{\eq}$, and $I_{\eq}$.\\ \\
\noindent We will denote the random variable whose distribution approximates the distribution of $S_{\eq}$ (resp. $-I_{\eq}$) with an $n$-fold convolution of gamma distributions by $\overline{G}_{n}$  (resp. $\underline{G}_{n}$). The notation $\varphi_{\overline{G}_n}(z)$ and $\psi_{\overline{G}_n}(z)$ is used for the MGF and CGF of $\overline{G}_{n}$ respectively, which, we recall, can both be derived via the $[n-1/n]$ Pad{\'e} Approximant of $(\psi_q^+)'(z)$ using the result of Proposition \ref{prop:ggc}. These will have the form
\begin{align*}
\varphi_{\overline{G}_n}(z) = \prod_{i=1}^{n}\left(1 - \frac{z}{\beta_i}\right)^{-\alpha_i}, \quad \psi_{\overline{G}_n}(z) = \log\left(\varphi_{\overline{G}_n}(z)\right),
\end{align*}
for positive constants $\{\alpha\}_{1\leq i \leq n}$ and $\{\beta\}_{1\leq i \leq n}$, which are easily extracted from the $[n-1/n]$ Pad{\'e} Approximant of $(\psi_q^+)'(z)$ by a partial fraction decomposition. The cumulative distribution function (CDF) of $\overline{G}_{n}$ is denoted $F_{\overline{G}_{n}}(x)$. We adopt the analogous notation for the MGF, CGF and CDF of $\underline{G}_{n}$.\\ \\
\noindent The random variable corresponding to the approximation based on a mixture of $n$ exponential distributions will be denoted $\overline{E}_{n}$ (resp. $\underline{E}_{n}$). The notation $\varphi_{\overline{E}_n}(z)$ and $\psi_{\overline{E}_n}(z)$ will be used for the MGF and CGF of $\overline{E}_{n}$, which are derived via the $[n-1/n]$ Pad{\'e} Approximant of $\varphi_q^+(z)$ using the result of Proposition \ref{prop:me}. These will have the form
\begin{align*}
\varphi_{\overline{E}_n}(z) = \sum_{i=1}^{n}\frac{\eta_i\omega_i}{\eta_i - z}, \quad \psi_{\overline{E}_n}(z) = \log\left(\varphi_{\overline{E}_n}(z)\right),
\end{align*}
where $\{\eta_i\}_{1\leq i \leq n}$ and $\{\omega_i\}_{1\leq i \leq n}$ are again positive constants obtained from the partial fraction decomposition of the $[n-1/n]$ Pad\'{e} Approximant of $\varphi^+_q(z)$. The CDF of $\overline{E}_n$ will be denoted $F_{\overline{E}_n}(x)$; analogous notation will be used for the MGF, CGF, and CDF of $\underline{E}_{n}$.
\subsection{Approximation of the CDF Two Ways}
\subsubsection{Cumulant Check}\label{subsubsec:cumchk}
As an initial test of the results of Section \ref{sect:main} we consider the cumulants of $X_{\eq}$, which we can calculate exactly via the CGF $\psi_{X_{\eq}}(z) := \log\left(q/(q - \psi_X(z))\right)$ whenever $q > 0$. Via the CGFs $\psi_q^+(z)$ and $\psi_q^-(z)$, derived in Corollaries \ref{cor:ggc} and \ref{cor:gegc}, and the methods of Section \ref{sect:techdets} we can also calculate the cumulants of $S_{\eq}$ and $-I_{\eq}$ exactly. If the results are correct, then the following relationship must hold for all $k \in \mathbb{N}$:
\begin{align}\label{eq:fundcheck}
\left.\frac{\text{d}^k}{\text{d}z^k}\psi_{X_{\eq}}(z) \right\vert_{z=0} =
\left. \frac{\text{d}^k }{\text{d}z^k}\left(\psi_q^+(z) + \psi_q^-(z)\right) \right\vert_{z=0}= \left. \frac{\text{d}^k }{\text{d}z^k}\left(\psi_q^+(z) + (-1)^k\psi_q^-(-z)\right) \right\vert_{z=0}.
\end{align}
Additionally, by construction, the relationship must also hold up to $k \leq 2n-1$ for the approximations based on the $[n-1/n]$ Pad\'{e} Approximant, i.e. we must also have
\begin{align}\label{eq:gamcheck}
\left.\frac{\text{d}^k}{\text{d}z^k}\psi_{X_{\eq}}(z) \right\vert_{z=0} = 
\left. \frac{\text{d}^k }{\text{d}z^k}\left(\psi_{\overline{G}_n}(z) + (-1)^k\psi_{\underline{G}_n}(z)\right) \right\vert_{z=0},\quad 1 \leq k \leq 2n-1,
\end{align}
whenever approximation by a gamma convolution is applicable. Additionally,
\begin{align}\label{eq:mecheck}
\left.\frac{\text{d}^k}{\text{d}z^k}\psi_{X_{\eq}}(z) \right\vert_{z=0}  = 
\left. \frac{\text{d}^k }{\text{d}z^k}\left(\psi_{\overline{E}_n}(z) + (-1)^k\psi_{\underline{E}_n}(z)\right) \right\vert_{z=0},\quad 1 \leq k \leq 2n-1
\end{align}
must always hold for the ME approximation. Note that while the identities \eqref{eq:fundcheck}, \eqref{eq:gamcheck}, and \eqref{eq:mecheck} are theoretically exact the degree of precision to which \eqref{eq:gamcheck}, and \eqref{eq:mecheck} hold when actually computed may depend on the level of precision we use in deriving the Pad\'{e} Approximants (see discussion at the end of Section \ref{sect:padstiel}).\\ \\
\noindent For this example, we consider the parameter set 
\begin{align*}
(\theta,\mu,\kappa,\sigma,q) = \left(-1,-4,\frac{187}{64},1,1\right),
\end{align*}
since, in this case, both the approximation by gamma convolution and by finite exponential mixture is applicable. In Table \ref{tab:check} we compute identities \eqref{eq:fundcheck} - \eqref{eq:mecheck} using this parameter set. The left-hand side of \eqref{eq:fundcheck} - \eqref{eq:mecheck} up to $k=10$ is displayed in the column labeled $\psi_{X_{\eq}}^{(k)}(0)$. In column \emph{Exact}, we compute the right-hand side of \eqref{eq:fundcheck}. As expected, the values match those in column $\psi_{X_{\eq}}^{(k)}(0)$ exactly. In column \emph{G 500} , we compute the right-hand side of \eqref{eq:gamcheck} for the random variables $\overline{G}_{5}$ and $\underline{G}_{5}$ using 500 digit precision to calculate the Pad\'{e} Approximant. We do the same for the right-hand side of \eqref{eq:mecheck} using the random variables $\overline{E}_{5}$ and $\underline{E}_{5}$ in column\emph{ E 500}. Note, we only use higher precision arithmetic to compute the Pad\'{e} Approximants, all other values are computed using 17 digits of precision. We see that at this level of precision, the approximations also satisfy the identities exactly. On a laptop with 4GB of memory and an Intel i5 CPU @ 2.27 GHz the entire computation to derive the parameters that define the distribution of $\overline{G}_5$ -- i.e. calculating the Taylor Series expansion of $(\psi^+_q)'(z)$, deriving the Pad\'{e} Approximation from this series, completing a partial fraction decomposition of the resulting function to isolate the parameters of the $n$-fold gamma convolution -- takes approximately 0.3 seconds. The same is true also for $\underline{G}_{5}$, $\overline{E}_{5}$ and $\underline{E}_{5}$. The code for this example is written using Mathematica.
\renewcommand{\arraystretch}{1.25}

\begin{table}
\begin{center}
\begin{small}
\begin{tabular}{ | c || c | c | c | c | c | c |}
\hline
$k$ & $\psi_{X_{\eq}}^{(k)}(0)$ & Exact & G 500 & E 500 \\
\hline\hline
1&-5.0000000000000000&-5.0000000000000000&-5.000000000000000&-5.000000000000000  \\
\hline
2&28.921875000000000&28.921875000000000&28.92187500000000&28.92187500000000 \\
\hline
3&-343.20581054687500&-343.20581054687500&-343.20581054687500&-343.20581054687500 \\
\hline
4&6196.8737068176270&6196.8737068176270&6196.8737068176270&6196.8737068176270 \\
\hline
5&-150452.69069820643&-150452.69069820643&-150452.69069820643&-150452.69069820643 \\
\hline
6&4.5921017309017433E6&4.5921017309017433E6&4.5921017309017433E6& 4.5921017309017433E6\\
\hline
7&-1.6888501187015734E8&-1.6888501187015734E8&-1.6888501187015734E8&-1.6888501187015734E8 \\
\hline
8&7.2689737036613218E9&7.2689737036613218E9&7.2689737036613218E9&7.2689737036613218E9 \\
\hline
9&-3.5843731491371288E11&-3.5843731491371288E11&-3.5843731491371288E11&-3.5843731491371288E11 \\
\hline
\end{tabular}
\end{small}
\caption{Computational results of identities \eqref{eq:fundcheck} - \eqref{eq:mecheck}}\label{tab:check}
\end{center}
\end{table}
\subsubsection{The CDF}
Continuing with the example in Section \ref{subsubsec:cumchk} we now consider the CDFs of $S_{\eq}$ and $-I_{\eq}$ ($F_{S_{\eq}}(x)$ and $F_{-I_{\eq}}(x)$ respectively) and the associated approximations. In Figure \ref{fig:cdfimage} we plot the $F_{S_{\eq}}(x)$ (blue $+$) and $F_{-I_{\eq}}(x)$ (green $\times$) as generated by a Monte Carlo simulation. That is, we simulate the exponential random variable $\eq$ and the discretize the interval $(0,\eq)$ using step sizes of $10^{-3}$. The random variables $X_{0.001}$ are simulated using the technique described in \cite{tinaturner} and summed to generate a path; the maximum (resp. minimum) along this path is taken as an approximation of a realization of $S_{\eq}$ (resp. $I_{\eq}$). The process is repeated $10^6$ times to generate an empirical CDF. \\ \\
\noindent Additionally plotted in Figure \ref{fig:cdfimage} are $F_{S_{\eq}}(x)$ (red $\circ$) and $F_{-I_{\eq}}(x)$ (fuchsia $\square$) as derived by numerical inversion of the Laplace Transforms $\varphi_q^{\pm}(z)$. The general technique we use for numerical inversion is described in Appendix A of \cite{lognorm}. Note, that for numerical inversion we are required to evaluate $\varphi_q^{\pm}(z)$ along a line $C + \iota\mathbb{R}$ in the complex plane, which means that we need to evaluate integrals of the form
\begin{align}\label{eq:lapint}
\int_{R}^{\infty}\log\left(\frac{u}{u-z}\right)\frac{Au + B}{(u-D)(u-E)\sqrt{(u-C)(u-R)}}\text{d}u,
\end{align}
numerically for $z$ on this line. This, however, does not pose a serious challenge: we make a change of variables $u \mapsto 2Ru^{-1} - 1$, such that the interval of integration becomes $(-1,1)$ and then use the Tanh-Sinh Quadrature as described in \cite{dbailes} (precision $p = 60$ and step size $h = 2^{-7}$ in the notation of \cite{dbailes}).\\ \\
In Figure \ref{fig:cdfimage2} along with $F_{S_{\eq}}(x)$ (red $\circ$) and $F_{-I_{\eq}}(x)$  (fuchsia $\square$) as computed by numerical inversion of the Laplace transform, we also plot $F_{\overline{E}_5}(x)$ (blue $+$) and $F_{\underline{E}_5}(x)$ (green $\times$). In Figure \ref{fig:cdfimage3} we make the same comparison using $F_{\overline{G}_5}(x)$ (blue $+$) and $F_{\underline{G}_5}(x)$ (green $\times$). Note that in the latter case, we also generate the CDF via numerical inversion of the Laplace transform -- this is much easier than computing $F_{S_{\eq}}(x)$ and $F_{-I_{\eq}}(x)$, however, since we do not have to calculate the integral \eqref{eq:lapint} at every step of the algorithm -- while in the former case we have explicit expressions for $F_{\overline{E}_5}(x)$ and $F_{\underline{E}_5}(x)$, e.g.
\begin{align*}
F_{\overline{E}_5}(x) = \sum_{i=1}^{5}\omega_i\left(1 - e^{-\eta_ix}\right).
\end{align*}
\\ \\
Thus from Figure \ref{fig:cdfimage} we get a numerical validation of our theoretical results and from Figures \ref{fig:cdfimage2} and \ref{fig:cdfimage3} we get a sense that both of our approximations work well, even when we employ a relative low degree Pad\'{e} approximant. To get a better sense of the quality of the approximation along the steep part of the CDF of $S_{\eq}$ near 0, which is not captured in Figures \ref{fig:cdfimage} - \ref{fig:cdfimage3}, we also plot the errors between $F_{S_{\eq}}(x)$ as computed by numerical Laplace inversion and $F_{\overline{E}_5}(x)$ (maroon $+$), $F_{\overline{E}_{10}}(x)$ (blue $+$), $F_{\overline{E}_{13}}(x)$ (green $+$), $F_{\overline{E}_{25}}(x)$ (purple $+$) and $F_{\overline{G}_5}(x)$ (black $\circ$) over the interval $(0.005,0.05)$ in Figure \ref{fig:cdfimage4}. We see that the lower degree ME approximations do not perform as well here and, depending on our desired level of accuracy, we may wish to choose a higher degree approximation (although $n=13$ already yields errors smaller than 0.005). By contrast, the fifth degree GGC approximation continues to perform well, which may justify the additional numerical effort required to compute the CDF in this case.
\begin{figure}[!htb]
\centering
\subfloat[]
{\label{fig:cdfimage}\includegraphics{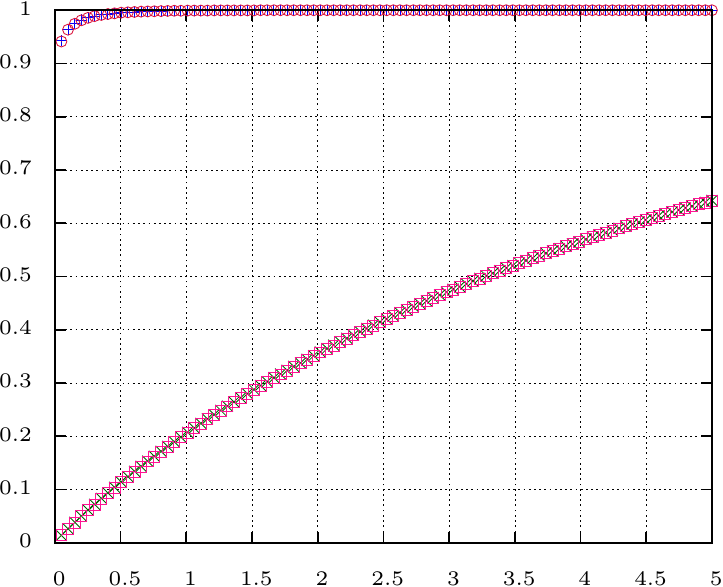}}
\subfloat[]
{\label{fig:cdfimage2}\includegraphics{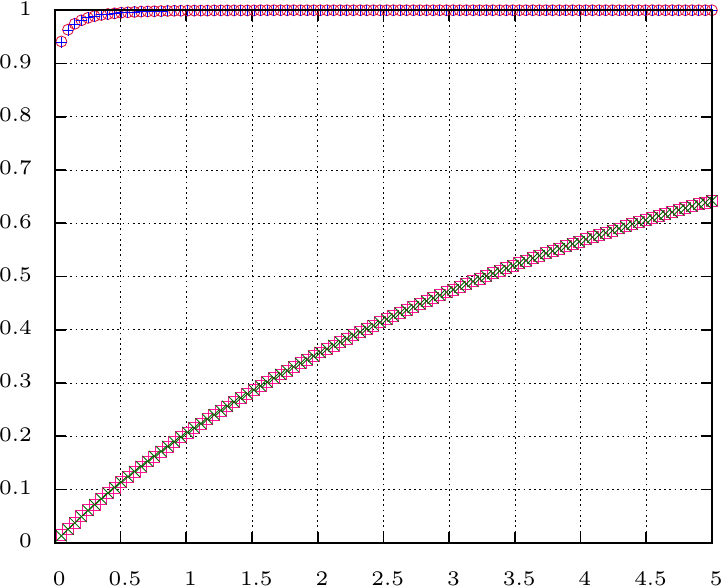}}\\
\subfloat[]
{\label{fig:cdfimage3}\includegraphics{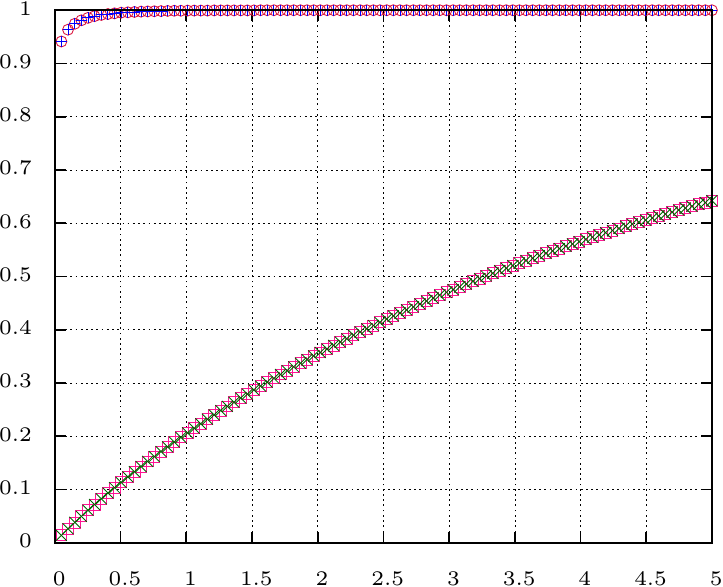}}
\subfloat[]
{\label{fig:cdfimage4}\includegraphics{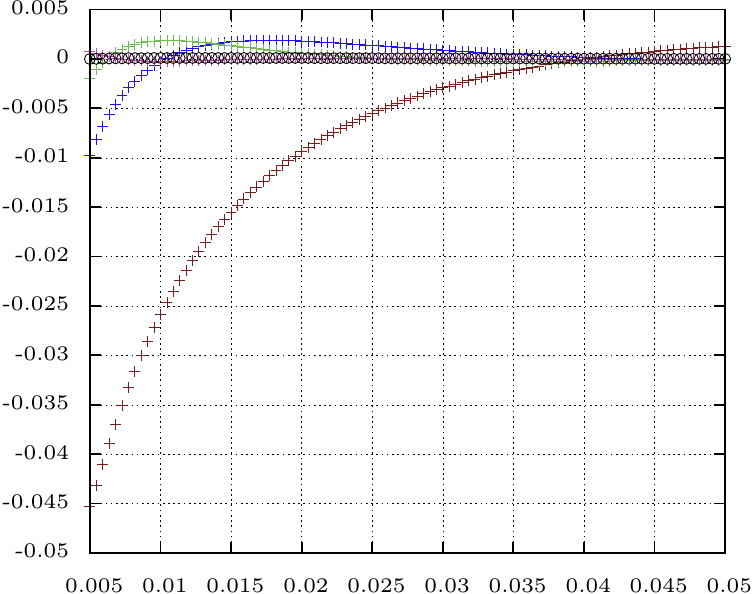}}
\caption{CDFs in (a) - (c), Errors in (d)}
\end{figure}
\subsection{Ruin Probabilities}\label{sect:ruin}
A simple application of the results to a financial problem, is to consider ruin probabilities. That is, we suppose that the capital of a company at time $t$ is modeled by $K_t = x + X_t$, where $x > 0$ is the initial capital and $X$ is a NIG process. The probability of ruin is then given by
\begin{align*}
R(x) := \p\left(\inf_{t > 0} K_t < 0\right) = \p(-I_{\infty} > x).
\end{align*}
Of course, this problem only makes sense when $-I_{\infty}$ is a.s. finite, i.e. when $\mu + \theta > 0$, a condition we assume from now on. The asymptotics of the ruin probability have been studied extensively in the context of insurance companies where $X$ was initially modeled as a compound Poisson process with only negative, exponentially distributed jumps. In this setting it was found that if \emph{Cram{\'e}r's condition} is satisfied, i.e. if $\psi_X(z)$ has a negative, real root $-\gamma$, then $R(x) \sim Ce^{-\gamma x}$, where $C$ is an explicitly defined constant. Doney and Bertoin \cite{donber} generalized this result, i.e. they showed the same asymptotics also apply when $X$ is a L\'evy process satisfying Cram\'er's condition plus an additional technical condition -- specifically that 0 is regular for $(0,\infty)$ -- which is satisfied by all NIG processes. A limiting formula for $C$ in this general setting was derived by Mordecki \cite{mordor}, who showed that
\begin{align}\label{eq:c}
C = \frac{\iota}{\gamma}\lim_{z \uparrow \iota\gamma}(z - \iota\gamma)\varphi_0^-(\iota z).
\end{align}
For the NIG process, with $-\gamma = \hat{\zeta}$ and the results of Corollary \ref{cor:q0}, we may then use \eqref{eq:c} to get an explicit representation of the asymptotics of $R(x)$, i.e. when $X$ is a NIG process satisfying $\mu + \theta > 0$ and Cram{\'e}r's condition we have
\begin{align}\label{eq:ruin}
R(x) \sim Ce^{\hat{\zeta} x},\quad\text{where} \quad C = \exp\left(\int_{\orp}\log\left(\frac{u}{u + \hat{\zeta}}\right)\tau^-_0(\text{d} u)\right),
\end{align}
and $\tau_0^-$ is the appropriate measure from Corollary \ref{cor:q0}. The integral in \eqref{eq:ruin} is easily evaluated numerically, for example via the Tanh-Sinh Quadrature (see \eqref{eq:lapint} and the discussion thereafter).\\ \\
\noindent Of course, we can also approximate $R(x)$ by deriving $F_{\underline{E}_n}(x)$ (or $F_{\underline{G}_n}(x)$ when appropriate, although in this case the ME representation seems more convenient). In particular,
\begin{align*}
R(x) \approx  1 - F_{\underline{E}_n}(x) = \sum_{i=1}^n{\omega_i}e^{-\eta_i x}.
\end{align*}
If we order $\{\eta_i\}_{1\leq i \leq n}$ and $\{\omega_i\}_{1\leq i \leq n}$ such that $\eta_1 \leq \eta_2 \leq \ldots \leq \eta_n$ then we would expect that $\eta_1 \approx -\hat{\zeta}$ and $C \approx \omega_1$ and that this approximation gets better with increasing $n$. \\ \\
\noindent Consider the parameter sets
\begin{align*}
\text{PS 1:}\quad (\theta,\mu,\kappa,\sigma,q) = \left(-1,\frac{3}{2},1,2,0\right)\quad\text{and}\quad\text{PS 2:}\quad (\theta,\mu,\kappa,\sigma,q) = \left(-1,4,\frac{1}{2},2,0\right),
\end{align*}
and note that for PS 1 we have $\hat{\zeta}(0) \neq \hat{\rho}$, whereas for $\hat{\zeta}(0) = \hat{\rho}$ for PS 2. In Table \ref{tab:asymp} we compute $\eta_1$ and $\omega_1$ for $n = 5, 10, 5, 25, 50$ and $75$. As a comparison we give the values of $-\hat{\zeta}$ and $C$, where the latter has been computed numerically with the Tanh-Sinh quadrature, in the row \emph{Exact}. We see that indeed $\eta_1$ and $\omega_1$ converge numerically to $-\hat{\zeta}$ and $C$ respectively. The convergence is slower for PS 2; this seems to reflect the fact that the condition $\hat{\zeta}(0) = \hat{\rho}$ is somewhat extreme.
\begin{table}
\begin{center}
\begin{small}
\begin{tabular}{ | c || c | c | c | c |}
\hline
{}& \multicolumn{2}{ c |}{PS 1} & \multicolumn{2}{ c |}{PS 2} \\
\hline
$n$ & $\eta_1$  & $\omega_1$ & $\eta_1$  & $\omega_1$ \\
\hline\hline
5&0.16000002709200613&0.73382866742186084&0.50109487544933153&0.66572495797628802  \\
\hline
10&0.16000000000000098&0.73382714607681802&0.50014426312102660 &0.62302276617409411 \\
\hline
15&0.16000000000000000&0.73382714607669872&0.50004356706493831&0.60879935656462980 \\
\hline
25&0.16000000000000000&0.73382714607669872&0.50000956018928113&0.59742364461027517 \\
\hline
50&0.16000000000000000&0.73382714607669872&0.50000120963128605&0.58889316511778638 \\
\hline
75&0.16000000000000000&0.73382714607669872&0.50000035988511168&0.58604984904352214 \\
\hline
\hline
Exact & 0.16 & 0.73382714607669872  & 0.5 & 0.58036339013109773\\
\hline
\end{tabular}
\end{small}
\caption{The values of $\eta_1$ and $\omega_1$ compared with exact values of $-\hat{\zeta}$ and $C$}\label{tab:asymp}
\end{center}
\end{table}
\subsection{Perpetual Options}\label{sect:ami}
As a more complex application in finance, let us consider the problem of pricing perpetual stock options under the assumption that the stock price $A$ at time $t$ has the form $A_t = A_0e^{X_t}$, where $A_0 > 0$ is the price at time $t=0$ and $X$ is L\'{e}vy process.  In \cite{mordorII}, Theorem 2, it is shown that the value of a perpetual put option under such a model is given by
\begin{align}\label{eq:put}
V = \frac{\e\left[\left(KC - A_0e^{I_{\er}}\right)^+\right]}{C},\quad C := \varphi_r^-(1),
\end{align}
where $r > 0$ is the interest rate, $K > 0$ is the strike price, and $x^+ = \max\{0,x\}$. Further, the option is optimally exercised at time
\begin{align*}
T = \inf\{t \geq 0: A_t \leq KC\}.
\end{align*}
Similar formulas are given for call options and the case where $r = 0$. \\ \\
\noindent If $X$ is taken to be a NIG process, then $C$ can be computed directly using the the results from this paper (see in particular Corollary \ref{cor:gegc} as well as \eqref{eq:lapint} and the discussion thereafter). That is, we can calculate the optimal exercise boundary exactly, and the value function can be approximated by
\begin{align}\label{eq:putopts}
V_n := \frac{\e\left[\left(KC - A_0e^{-\underline{E}_n}\right)^+\right]}{C} 
= \begin{dcases}
\sum_{i=1}^{n}\omega_i\left(\frac{C}{A_0}\right)^{\eta_i} \frac{K^{\eta_i + 1}}{1 + \eta_i} &,\;\; \log\left(\frac{CK}{A_0}\right)  < 0 \\
\sum_{i=1}^{n}\omega_i\left(K - \frac{A_0}{C}\frac{\eta_i}{1 + \eta_i}	\right) &,\;\; \log\left(\frac{CK}{A_0}\right)  \geq 0	
\end{dcases}	,																	
\end{align}
where we have used the fact that $\underline{E}_n$ has a density of the form $f_{\underline{E}_n}(x) = \sum_{i=1}^n\omega_i\eta_ie^{-\eta_i x}$.\\ \\
Consider the parameter set
\begin{align*}
\quad (\theta,\mu,\kappa,\sigma,r,K) = \left(-1,0.723914,1,0.25,0.01,100\right),
\end{align*}
and observe that the parameters have been chosen such that $\psi_X(1) = r$, or equivalently that $e^{-rt}A_t$ is a martingale, i.e. that we are working with a risk neutral martingale measure. 
%With this set of parameters we compute $V$ also by Monte Carlo simulation with $10^6$ iterations and discretization step size $10^{-2}$. This information is summarized in Table \ref{tab:opt}. 
Using this parameter set, we calculate $V_n$ for various strikes and values of $n$; the results are summarized in Table \ref{tab:opt}. We see that the price converges numerically very rapidly and, in fact, is likely already good enough with $n=3$.
\renewcommand{\arraystretch}{1.3}
\begin{table}
\begin{center}
\begin{tabular}{ | c || c | c | c | c | c |}
  \hline 
\diagbox{$n$}{$A_0$} & 5 & 50 & 100 & 150 & 195 \\
\hline \hline 
3 & 95.010756 & 87.212858 & 85.163045 & 83.990865 & 83.242228\\
\hline
5 & 95.000051 & 87.205429 & 85.158933 & 83.988238 & 83.240350\\
\hline
7 & 95.000000 & 87.205790 & 85.158900 & 83.988135 & 83.240238 \\
\hline
9 & 95.000000 & 87.205757 & 85.158913 & 83.988149 & 83.240249\\
\hline
11 & 95.000000 & 87.205763 & 85.158911 & 83.988147 & 83.240248 \\
\hline
75 & 95.000000 & 87.205762 & 85.158911 & 83.988147 & 83.240248 \\
%\hline
%\hline
%MC & 95.000000 & 87.182607 & 85.129793 & 83.962736 & 83.218793 \\
\hline
\end{tabular}
\caption{$V_n$ calculated for different values of $n$ and $A_0$.}\label{tab:opt}
\end{center}
\end{table}
%\begin{remark} 
%It is perhaps worthwhile to compare the given approach
%\begin{itemize}
%\item Direct, formulas for constants like $C$.
%\item No approximation of an infinite activity, infinite variation process by a finite activity, finite variation process.
%\item Easier - Steps: calculate moments of representing measure (numerically because measure not so easy to deal with, no obvious connection to known orthogonal polynomials), calculate roots (numerically)
%\item Other approach probably easier to adapt for complex $q$.
%\end{itemize}
%\demo
%\end{remark}
\clearpage
\appendix
\appendixpage
\section{Solutions of $\psi_X(z) = q$}\label{sec:appendix}
Recall that the Laplace exponent of an NIG process $X$ has the form
\begin{align}\label{eq:apppsi}
\psi_X(s) =  \frac{1}{\kappa} - \frac{1}{\kappa}\sqrt{1 - 2\kappa\theta z - \kappa\sigma^2 z^2} + \mu z.
\end{align}
where $q,\,\kappa,\,\sigma \in \orp$ and $\mu,\,\theta \in \mathbb{R}$. In this appendix we prove some basic facts about the solutions of the equation
\begin{align}\label{eq:appmain}
\psi_X(z) = q,
\end{align}
 which together prove the statements of Proposition \ref{prop:themostannoying}.
First, we recall the definitions
\begin{align*}
\rho := \frac{-\theta + \sqrt{\theta^2 + \frac{\sigma^2}{\kappa}}}{\sigma^2},\qquad \hat\rho := \frac{-\theta - \sqrt{\theta^2 + \frac{\sigma^2}{\kappa}}}{\sigma^2},
\end{align*}
such that $p(z) := (1 - z/\rho)(1 - z/\hat{\rho}) = 1 - 2\kappa\theta z - \kappa\sigma^2z^2$. 
\begin{lemma}\label{lem:obvious}
If $z_0 \in \mathbb{C}$ is a solution of \eqref{eq:appmain} then $z_0 = \zeta$ or $z_0 = \hat{\zeta}$, where
\begin{gather}\label{eq:themroots}
\zeta := \frac{-\theta -\mu +\kappa  \mu  q +\sqrt{d}}{\kappa  \mu ^2+\sigma ^2}, \quad \hat\zeta :=\frac{-\theta -\mu +\kappa  \mu  q -\sqrt{d}}{\kappa  \mu ^2+\sigma ^2}, \\
d := \theta ^2+\mu ^2-2 \theta  \mu  (q\kappa-1)+q \sigma ^2 (2-q\kappa),\nonumber
\end{gather}
and
\begin{align*}
z_0 \in \mathbb{V} := \{z \in \mathbb{C}\,:\,q - 1/\kappa \leq \mu\re(z)\}.
\end{align*}
\end{lemma}
\begin{proof}
Rewriting \eqref{eq:appmain} as
\begin{align}\label{eq:prequad}
r(z) = \sqrt{p(z)},
\end{align}
where $r(z) := \mu\kappa z + 1 - q\kappa$ shows that if $z_0 \notin \mathbb{V}$, then $z_0$ cannot be a solution of \eqref{eq:appmain}, since the positive square root function maps $\mathbb{C}$ to the right half of the complex plane. If, however, $z_0$ does solve \eqref{eq:appmain}, then it also solves the associated quadratic equation, which we get by squaring both sides of \eqref{eq:prequad}. The solutions of the quadratic equation have the form \eqref{eq:themroots}.
\end{proof}
\noindent Together the requirement that $z_0 = \zeta$ or $z_0 = \hat{\zeta}$ and $z_0 \in \mathbb{V}$ will be referred to as Condition A from here on. 
\begin{lemma}\label{lem:poslem}
If $z_0$ satisfies \eqref{eq:appmain}, then $z_0 \in [\hat{\rho},0) \cup (0,\rho]$ and $d > 0$.
\end{lemma}
\begin{proof}
Note that since  $\psi_X(0) = 0$ and $q > 0$ we cannot have $z_0 =0$. For the proof that $z_0$ is real with $d > 0$, we reduce the problem by considering cases for the variables $\mu$ and $\theta$.
 \\ \\
\underline{$\mu = 0$} \\ \\
Condition A reduces to $z_0$ equal to one of
\begin{align}\label{eq:thetas}
\zeta = \frac{-\theta + \sqrt{\theta^2 + q(2-q\kappa)\sigma^2}}{\sigma^2},\qquad \hat\zeta = \frac{-\theta - \sqrt{\theta^2 + q(2-q\kappa)\sigma^2}}{\sigma^2},
\end{align}
and $q \leq 1/\kappa$, i.e. the result is immediate.\\ \\
\underline{$\theta = 0$} \\ \\
In this case the formulas for $\zeta$ and $\hat{\zeta}$ reduce to
\begin{align*}
\zeta = \frac{-\mu +\kappa  \mu  q +\sqrt{\mu ^2+q \sigma ^2 (2-q\kappa)}}{\kappa  \mu ^2+\sigma ^2},\qquad \hat\zeta =\frac{-\mu +\kappa  \mu  q -\sqrt{\mu ^2+q \sigma ^2 (2-q\kappa)}}{\kappa  \mu ^2+\sigma ^2}.
\end{align*}
If we assume that $d = \mu ^2+q \sigma ^2 (2-q\kappa) \leq 0$, then also $(2-q\kappa) \leq 0$. However, under this assumption, the second part of Condition A reduces to 
\begin{align}
q-\frac{1}{\kappa }\leq \frac{\mu  ( \kappa  \mu  q -\mu)}{\kappa  \mu ^2+\sigma ^2} & \Leftrightarrow q  -\frac{1}{\kappa } \leq 0,
\end{align}
so that we arrive at a contradiction.\\ \\
\underline{$\mu > 0$ and $\theta > 0$} \\ \\
If we assume that $d =\theta ^2+\mu ^2-2 \theta  \mu  (q\kappa-1)+q \sigma ^2 (2-q\kappa) \leq 0$, then from Condition A we must have
\begin{align}\label{eq:muthetabigger}
q-\frac{1}{\kappa }\leq \frac{\mu  (\kappa  \mu  q -\theta -\mu )}{\kappa  \mu ^2+\sigma ^2}\quad\Leftrightarrow\quad\theta\mu + \sigma^2\left(q - \frac{1}{\kappa}\right) \leq 0.
\end{align}
From this it follows that we must have $q\kappa < 1$. However, rewriting the inequality for $d$ as
\begin{align*} 
d = \theta^2 + \mu^2 + q\sigma^2 - (q\kappa - 1)(2\theta\mu + q\sigma^2) \leq 0
\end{align*} 
implies that $q\kappa  > 1$, which is a contradiction.\\ \\
\underline{$\mu < 0$ and $\theta < 0$} \\ \\
Proof identical to the case $\mu > 0$ and $\theta > 0$. \\ \\
\underline{$\mu > 0$ and $\theta < 0$} \\ \\
We assume again for contradiction that $d = \theta ^2+\mu ^2-2 \theta  \mu  (q\kappa-1)+q \sigma ^2 (2-q\kappa) \leq 0$. It is clear that if $q\kappa \leq 2$, then we must also have $q\kappa< 1$, otherwise $d$ will certainly be greater than zero and the contradiction is immediate. However, if $ q\kappa < 1$, then
\begin{align*}
\theta ^2+\mu ^2-2 \theta  \mu  (q\kappa-1)+q \sigma ^2 (2-q\kappa) > \theta ^2+\mu ^2 + 2 \theta  \mu+q \sigma ^2 (2-q\kappa) = (\theta +\mu)^2 + q \sigma ^2 (2-q\kappa) > 0;
\end{align*}
it follows that $q\kappa$ must be greater than 2. Under this assumption, we rearrange the inequality $d \leq 0$ and the inequality \eqref{eq:muthetabigger} to get
\begin{align*}
\sigma^2 \geq \frac{\theta ^2+2 \theta  \mu +\mu ^2-2 \theta  \kappa  \mu  q}{q (q\kappa-2)}\quad\text{ and }\quad \sigma^2 \leq \frac{\theta  \kappa  \mu }{1-q\kappa}.
\end{align*}
Rewriting the right hand side of the first inequality as
\begin{align*}
\frac{\theta ^2+2 \theta  \mu +\mu ^2-2 \theta  \kappa  \mu  q}{q (q\kappa-2)} = \frac{(\theta + \mu)^2}{q(q\kappa - 2)} + 2\left(\frac{q\kappa - 1}{q\kappa -2}\right)\left(\frac{\theta\kappa\mu}{1 - q\kappa}\right) > \frac{\theta\kappa\mu}{1 - q\kappa}
\end{align*}
shows that we have once again arrived at a contradiction.\\ \\
\underline{$\mu < 0$ and $\theta > 0$} \\ \\
Proof identical to the case $\mu > 0$ and $\theta < 0$.\\ \\
Finally, to show that $\hat{\rho} \leq z_0 \leq \rho$ we rewrite the equation $\psi_X(z) = q$ as
\begin{align}\label{eq:rewri}
1 -q\kappa  + \kappa \mu z = \sqrt{\left(1 - \frac{z}{\rho}\right)\left(1 - \frac{z}{\hat{\rho}}\right)}.
\end{align}
We see that if we had $z_0 < \hat{\rho}$ or $\rho < z_0$, then the left hand side of \eqref{eq:rewri} would be a real number, whereas the right hand side of \eqref{eq:rewri} would be purely imaginary number, i.e. the equality would not hold so that $z_0$ could not be a solution of \eqref{eq:appmain}.
\end{proof}
\noindent The following result then follows almost immediately from Lemmas \ref{lem:obvious} and \ref{lem:poslem}.
\begin{lemma}\label{lem:tietogether}
A number $z_0$ satisfies \eqref{eq:appmain} iff $z_0 \in [\hat{\rho},0)\cup(0,\rho]$, $z_0 = \zeta$ or $z_0 = \hat{\zeta}$, and $q - \frac{1}{\kappa} \leq \mu z_0$. 
\end{lemma}
\begin{proof}\ \\
\noindent($\Rightarrow$)$\quad$ This direction is proven in Lemmas \ref{lem:obvious} and \ref{lem:poslem}.\ \\ \\
\noindent ($\Leftarrow$)$\quad$ The assumption that $z_0 = \zeta$ or $z_0 = \hat{\zeta}$ implies
\begin{align}\label{eq:oncemore}
[r(z_0)]^2 = (\mu\kappa z_0 + 1 - q\kappa)^2 = p(z_0).
\end{align}
Further, since $z_0$ is assumed to be real such that $q - \frac{1}{\kappa} \leq \mu z_0$, taking the square root of both sides of \eqref{eq:oncemore} yields $\mu\kappa z_0 + 1 - q\kappa$ on the right-hand side, i.e. $z_0$ must solve \eqref{eq:appmain}.
\end{proof}
\begin{lemma}\label{lem:notwo}
If $\zeta$ (resp. $\hat{\zeta}$) satisfies \eqref{eq:appmain}, then  $0 < \zeta$ (resp. $\hat{\zeta}< 0$).
\end{lemma}
\begin{proof}
Note that $\zeta <0$ implies that $-\theta + (q\kappa-1)\mu < 0$ and that
\begin{align*}
(\kappa  \mu  q -\theta -\mu )^2 - (\theta ^2+\mu ^2-2 \theta  \mu  (q\kappa-1)+q \sigma ^2 (2-q\kappa)) > 0\quad \Leftrightarrow \quad q (q\kappa-2) \left(\kappa  \mu ^2+\sigma ^2\right) > 0,
\end{align*}
where the latter statement is equivalent to the requirement that $kq > 2$. Recall also that if $\zeta$ satisfies \eqref{eq:appmain} we must have from Condition A that
\begin{align}
\kappa\theta\mu + (q\kappa - 1)\sigma^2 &\leq \kappa\mu\sqrt{\theta ^2+\mu ^2-2 \theta  \mu  (q\kappa-1)+q \sigma ^2 (2-q\kappa)}.\label{eq:keyineq}
\end{align}
Using these statements we consider various cases for $\mu$ and show that the assumption $\zeta$ satisfies \eqref{eq:appmain} and $\zeta < 0$ leads to a contradiction in each case.\\ \\
\underline{$\mu = 0$} \\ \\
Under the assumptions \eqref{eq:keyineq} results in the inequality $(q\kappa - 1)\sigma^2 \leq 0$, which is a contradiction since we have shown that $kq > 2$.\\ \\
\underline {$\mu < 0$ and $\theta \leq 0$} \\ \\
Similar to the case $\mu = 0$, we will have strictly positive quantity on the left-hand side of the inequality \eqref{eq:keyineq} whereas the right-hand side is at most zero.\\ \\
\underline {$\mu < 0$ and $\theta > 0$} \\ \\
Since $\mu < 0$ we must have
\begin{align*}
\kappa\theta\mu + (q\kappa - 1)\sigma^2 \leq 0\quad\Leftrightarrow\quad \frac{(1-q\kappa)\sigma^2}{\kappa \mu} \leq \theta
\end{align*}
in order for \eqref{eq:keyineq} to hold. However squaring both sides of \eqref{eq:keyineq} and solving for $\theta$ yields, after some algebra,
\begin{align}\label{eq:oneway}
\theta \leq \frac{\mu}{2(q\kappa - 1)} + \frac{\sigma^2(1- q\kappa)}{2\kappa\mu}\quad \Rightarrow\quad \theta < \frac{(1-q\kappa)\sigma^2}{\kappa \mu},
\end{align}
so that we arrive once more at a contraction.
\\ \\
\underline {$\mu > 0$ and $\theta \leq 0$} \\ \\
We have
\begin{align}\label{eq:onemorecond}
\zeta < 0\quad\Leftrightarrow\quad -\theta + \mu(q\kappa -1) < -\sqrt{\theta ^2+\mu ^2-2 \theta  \mu  (q\kappa-1)+q \sigma ^2 (2-q\kappa)},
\end{align}
which is a contradiction because the left-hand side of the above inequality is a strictly positive number and the right-hand side is at most zero.
\\ \\
\underline {$\mu > 0$ and $\theta > 0$} \\ \\
From \eqref{eq:onemorecond} we have
\begin{align*}
 -\theta + \mu(q\kappa -1) < 0 \quad \Leftrightarrow \quad \mu(q\kappa -1) < \theta,
\end{align*}
but squaring both sides of \eqref{eq:keyineq} and solving for $\theta$ yields,   
\begin{align}\label{eq:anotherway}
\theta \leq \mu(qk -1) + \frac{\mu(1 - 2(q\kappa - 1)^2)}{2(q\kappa -1) } + \frac{\sigma^2(1-q\kappa)}{2\kappa \mu} \quad \Rightarrow \quad \theta < \mu(qk -1),
\end{align}
which is a contradiction. Note: To reconcile the formulas in \eqref{eq:oneway} and \eqref{eq:anotherway} simply add an subtract $\mu(qk -1)$ on the right-hand side of the first inequality in \eqref{eq:oneway}.\\ \\
The proof for $\hat{\zeta} < 0$ follows from identical arguments.
\end{proof}
\begin{lemma}\label{lem:multiplicity}
If $z_0$ satisfies \eqref{eq:appmain} and $\hat{\rho} < z_0 < \rho$, then $z_0$ is a simple zero of $q - \psi_X(z)$.
\end{lemma}
\begin{proof}
From their definition \eqref{eq:themroots} and Lemma \ref{lem:poslem} it is clear that if at least one of $\zeta$ or $\hat{\zeta}$ is a solution of \eqref{eq:appmain}, then $\zeta \neq \hat{\zeta}$, i.e. $p(z) - [r(z)]^2$ has no zeros of multiplicity two when at least one of its zeros is a solution of $q = \psi_X(z)$. Suppose $z_0 \in (\hat{\rho},\rho)$ is a zero of multiplicity two of $q - \psi_X(z)$. Then by definition, there is an open ball $B$ around $z_0$ and a function $f(z)$ analytic on $B$ such that $f(z_0) \neq 0$ and 
\begin{align}\label{eq:notwozeros}
\kappa(q - \psi_X(z)) = \sqrt{p(z)} - r(z)= (z - z_0)^2f(z),\quad z \in B.
\end{align}
However, multiplying both sides of \eqref{eq:notwozeros} by the analytic (on $B$) function $\sqrt{p(z)} + r(z)$ yields
\begin{align}
p(z) - [r(z)]^2 = (z - z_0)^2f(z)\left(\sqrt{p(z)} + r(z)\right).
\end{align}
Since $z_0$ is a zero of both $\sqrt{p(z)} - r(z)$ and $\sqrt{p(z)} + r(z)$ only if $z_0 = \rho$ or $z_0 = \hat{\rho}$, the preceding shows that $p(z) - [r(z)]^2$ can be factored into the product of $(z - z_0)^2$ and the function $f(z)\left(\sqrt{p(z)} + r(z)\right)$, which is analytic on $B$ and non-zero at $z_0$. That is, $z_0$ must also be a zero of multiplicity two for $p(z) - [r(z)]^2$, which is a contradiction.
\end{proof}
\begin{lemma}\label{lem:notpastrho}
If $\zeta \in \mathbb{R}$ (resp. $\hat{\zeta} \in \mathbb{R}$) then $\hat{\rho} \leq \zeta \leq \rho$ (resp. $\hat{\rho} \leq \hat{\zeta} \leq \rho$).
\end{lemma}
\begin{proof}
By definition $\zeta$ is a solution of 
\begin{align}\label{eq:assoncemore}
p(z) = \left(1 - \frac{z}{\rho}\right)\left(1 - \frac{z}{\hat{\rho}}\right) = [r(z)]^2.
\end{align}
If $\zeta \in \mathbb{R}$ and $\zeta > \rho$ or $\zeta < \hat{\rho}$, then the left-hand side of the above equation, when evaluated at $\zeta$, yields a strictly negative number, and the right-hand evaluated at $\zeta$ yields a number that is greater than or equal to zero. That is, $\zeta$ does not satisfy \eqref{eq:assoncemore}, i.e. we have arrived at a contradiction. The same exercise can be repeated with $\hat{\zeta}$ and yields the same conclusion.
\end{proof}
\begin{lemma}\label{lem:multiplicity}
Neither $\rho = \hat{\zeta}$ nor $\hat{\rho} = \zeta$ is possible.
\end{lemma}
\begin{proof}
By definition, $\zeta$ and $\hat{\zeta}$ are either both real or both have nonzero imaginary part. In the latter case, the result follows immediately. If they are real, it is clear that $\hat{\zeta} < \zeta$. Therefore, the assumption that either $\rho = \hat{\zeta}$ or $\hat{\rho} = \zeta$ contradicts the result of Lemma \ref{lem:notpastrho}, which stipulates that both $\rho$ and $\hat{\rho}$ lie in the interval $[\hat{\rho},\rho]$.
\end{proof}
\begin{lemma}
Both $\rho = \zeta$ and $\hat{\rho} = \hat{\zeta}$ iff $\mu = 0$ and $q = 1/\kappa$.
\end{lemma}
\begin{proof}\ \\
\noindent($\Rightarrow$)$\quad$ Since both $\rho$ and $\hat{\rho}$ satifsy the associated quadratic equation $p(z) = [r(z)]^2$, we may plug these values in to get the following system of equations
\begin{align*}
\mu\kappa\rho + 1 - q\kappa = 0\quad\text{ and }\quad \mu\kappa\hat{\rho} + 1 - q\kappa = 0,
\end{align*}
from which it follows that $\mu = q - 1/\kappa = 0$.\ \\ \\
\noindent($\Leftarrow$)$\quad$ If $\mu = q - 1/\kappa = 0$ then the associated quadratic equation $p(z) = [r(z)]^2$ reduces to
\begin{align*}
0 = \left(1 - \frac{z}{\rho}\right)\left(1 - \frac{z}{\hat{\rho}}\right),
\end{align*}
from which it is clear that $\rho = \zeta$ and $\hat{\rho} = \hat{\zeta}$.
\end{proof}
\clearpage
\bibliographystyle{plain}
\bibliography{references}
\end{document}